\documentclass[final,1p,times]{elsarticle}

\usepackage{amsmath,amssymb,amsthm}
\usepackage{enumerate}
\usepackage{geometry}
\usepackage{graphicx}
\usepackage{subfigure}
\usepackage{epsfig}
\usepackage{graphicx}
\usepackage{colordvi}
\usepackage{graphics}
\usepackage{float}
\usepackage{dutchcal}
\usepackage{xcolor}

\numberwithin{equation}{section}
\newtheorem{theorem}[equation]{Theorem}
\newtheorem{cor}[equation]{Corollary}
\newtheorem{lemma}[equation]{Lemma}
\newtheorem{remark}[equation]{Remark}

\newtheorem{algorithm}{Algorithm}[section]
\newtheorem{assumption}{Assumption}[section]
\newtheorem{definition}[equation]{Definition}

\def \bX{{\mathbf X}}

\def \bV{{\mathbf V}}
\def \bH{{\mathbf H}}

\def \bu{{\mathbf u}}
\def \bs{{\mathbf s}}

\def \blf{{\mathbf f}}

\newcommand{\bphi}{{\boldsymbol \phi}}
\newcommand{\btau}{{\boldsymbol \tau}}
\newcommand{\bdiv}{{\boldsymbol \div}}
\newcommand{\bleta}{{\boldsymbol \eta}}
\newcommand{\taus}{{\tau_s}}
\newcommand{\eps}{{\varepsilon}}

\def \bw{{\mathbf w}}

\def \be{{\mathbf e}}

\def \bh{{\mathbf h}}

\newcommand{\bv}{{\bf v}}
\newcommand{\bg}{{\bf g}}
\newcommand{\bx}{{\bf x}}

\def\div{\operatorname{div}}
\DeclareMathOperator{\argmin}{argmin}

\newcommand{\nr}[1]{\ensuremath{\left\|{#1}\right\|}}

\usepackage{colordvi}

\journal{arXiv}

\begin{document}
	
	\begin{frontmatter}
		
		

		
		\title{Anderson acceleration for a regularized Bingham model}
		
		\author[label2]{Sara Pollock}
		\ead{s.pollock@ufl.edu}
		\address[label2]{Department of Mathematics, University of Florida, Gainesville, FL 32611, USA}
		\author[label1]{Leo G. Rebholz\corref{cor1}\fnref{lab1}}
		\ead{rebholz@clemson.edu}
		\author[label1]{Duygu Vargun\fnref{lab1}}
		\ead{dvargun@clemson.edu}
		\address[label1]{Department of Mathematical Sciences, Clemson University, Clemson, SC 29634, USA}
		
		\cortext[cor1]{Corresponding  author.}
		
		\begin{abstract}
			This paper studies a finite element discretization of the regularized Bingham equations that describe viscoplastic flow. 
			An efficient nonlinear solver for the discrete model is then proposed and analyzed.  The solver is based
			on Anderson acceleration (AA) applied to a Picard iteration, and we show accelerated convergence of the method by applying AA theory (recently developed by the authors) to
			the iteration, after showing sufficient smoothness properties of the associated fixed point operator.  Numerical tests of spatial convergence are provided, as are results of 
			the model for 2D and 3D driven cavity simulations.  For each numerical test, the proposed nonlinear solver is also tested and shown to be very effective and robust with respect to the regularization parameter as it goes to zero. 
		\end{abstract}
		
		
		
		
	\end{frontmatter}
	
	
	
	\section{Introduction}\label{introsection}
	A Bingham plastic is a material that as a solid at lower shear stress but flows with a constant viscosity when larger shear stress is applied \cite{B22}. Fresh concrete, dough, blood in the capillaries, muds, toothpaste, and ketchup are a few examples of such materials \cite{DGG07}. These applications motivate researchers from several fields and industries to study their behavior, their mathematical formulations and 
	properties, and to develop software to perform simulations \cite{BDY83}.
	
	The governing equation of Bingham plastics is given by
	\begin{equation}\label{binghampde}
		\begin{aligned}
			-\bdiv \btau + \nabla p &= \blf\\
			\nabla\cdot \bu &= 0
		\end{aligned}
	\end{equation}
	in a bounded and connected domain $\Omega\subset\mathbb{R}^d,\ d=2,3$.
	Here, $\btau$ is the $\bu$ is the fluid velocity and $p$  is the pressure.
	
	The change in the behavior of Bingham plastic occurs after the applied stress $\tau$ exceeds a certain threshold called the yield stress, which is denoted by $\taus$. {\color{black} $D\bu=\frac{1}{2}[\nabla D\bu + \nabla D\bu^T]$ is strain rate tensor defined as a symmetric part of velocity gradient and $|D\bu|=\sqrt{D\bu:D\bu}$ is the Frobenius norm of $D\bu$. If $|D\bu|\neq 0$, then $|\btau|=|2\mu D\bu + \taus \frac{D\bu}{|D\bu|}|>\taus$ and so Bingham plastic behaves a as fluid and the equations describing the flow of a Bingham plastic are given by
		\begin{equation}\label{binghampde2}
		\begin{aligned}
			-\bdiv \left( \hat \mu D\bu \right) + \nabla p &= \blf,\\
			\nabla\cdot \bu &= 0,
		\end{aligned}
	\end{equation}
where $\hat \mu= 2\mu+\frac{\taus}{|D\bu|}$ is the shear-dependent viscosity with given problem dependent constants which are the plastic viscosity $\mu>0$ and  the yield stress $\taus\geq 0$,

On the other hand, if $|D\bu|= 0$, then $|\btau|\leq\taus$ and so Bingham plastic behaves as solid material, and the equation (\ref{binghampde}) is not valid. Hence, the domain of Bingham plastics can be split into two subdomains, determined by relationship $\btau$ and $\taus$, and this relation can be rewritten as
	\begin{align*}
	D\bu=  \left\{
	\begin{array}{ll}
		0 & \text{for}\ |\btau|\leq \taus\ (\text{rigid region}:\Omega_r),\\
		\left(1-\frac{\taus}{|\btau|}\right) \frac{\btau}{2\mu}& \text{for}\ |\btau|> \taus\ (\text{fluid region}:\Omega_f).
	\end{array} 
	\right.
\end{align*}
In the fluid region $\Omega_f$, Bingham plastics behave like a fluid, and the equations (\ref{binghampde}) turn into \eqref{binghampde2} and they can be viewed as a generalization of the Stokes equations having a shear-dependent viscosity $\hat{\mu}$. In the case of no yield stress; i.e. $\taus=0$, (\ref{binghampde}) reduces exactly to Stokes equations with constant viscosity $\mu$. In the rigid (or plug) region $\Omega_r$, Bingham plastics behave like a solid, the equation (\ref{binghampde}) cannot describe the any motion of the solid material.
}

There are two major difficulties associated with solving the Bingham equations: the interface between the rigid and fluid regions is not known a priori, and $\hat{\mu}$ becomes singular in the rigid region since $|D\bu| = 0$. There are two main approaches to handle these difficulties. One is to reformulate the problem as a variational inequality \cite{DL76,G84} and to use either operator splitting methods \cite{DGG07, HK15, S98} or an augmented Lagrangian approach \cite{Z10, GT89}. The other is to introduce regularization for {\color{black}$\btau$,}
which circumvents both issues but introduces a consistency error. The most common types of regularization are proposed by Papanastasiou \cite{P87} and Bercovier-Engelmann \cite{BE1980}, but other types have been considered \cite{M07, TM83}.
	
	We consider the Bercovier-Engelman regularization, in which $|D\bu|$ is replaced by $|D\bu|_{\varepsilon}=\sqrt{D\bu:D\bu+\varepsilon^2}$ in the shear-dependent viscosity $\hat{\mu}$ {\color{black}in \eqref{binghampde2}}, with $\varepsilon$ denoting the regularization parameter. 
	This regularized formulation provides one nonsingular system for {\color{black}the entire domain} $\Omega$:
	\begin{equation}\label{regeqn}
		\begin{aligned}
			-\div\left(2\mu+\frac{\tau_s}{|D\bu|_{\varepsilon}} \right)D\bu+\nabla p= \blf,\\
			-\nabla\cdot\bu=0.
		\end{aligned}
	\end{equation}
	With \eqref{regeqn}, the entire domain is treated computationally as a single region.
	The approximated plug region can be recovered by inspecting regions of high viscosity. 
	However, there is an obvious drawback in that any regularization affects the accuracy of results due to physical inconsistency.  As is known from \cite{GO09}, the regularized problem (\ref{regeqn}) provides an approximate solution for non-regularized Bingham problem (\ref{binghampde}) which satisfies only
	\begin{align*}
		\|D(\bu-\bu_{nonreg})\|\leq C\sqrt{\eps}
	\end{align*}
	where $\bu$ is the solution of  (\ref{regeqn}) and $\bu_{nonreg}$ is the solution of  (\ref{binghampde}).  Thus not surprisingly, one must choose small $\eps$ for good accuracy  \cite{DG02, S2000, FN05}.  Unfortunately, as we discuss below, small $\eps$ causes solvers to fail.  The purpose of this paper is to propose a method that is both accurate and also robust for small $\eps$.
	
	Nonlinear solvers used for solving the regularized Bingham model are typically iterative schemes of Newton or Picard type, however there are drawbacks with both of these approaches.
	{\color{black} The issue with the standard Picard iteration is that convergence is slow and may not be guaranteed, 
	especially as $\eps$ gets small \cite{AHOV11}.   Convergence can be improved by introducing an auxiliary tensor variable as in \cite{AHOV11}, however this makes solving the linear systems at each iteration more difficult.  Using a Newton iteration instead of Picard can 
	provide quadratic convergence, but at the expense of more difficult linear system solves at each iteration.  Moreover, Newton's domain of convergence is not robust with respect to $\eps$ (see \cite{DG02} and numerical results in \cite{GO09,HOT03}).  We note that in general, analytical convergence results for iterative solvers for regularized Bingham is lacking in the literature. }
	
	One aim in this paper is to improve the Picard iteration for the regularized Bingham problem \eqref{regeqn} considered in \cite{AHOV11} by enhancing it with Anderson acceleration (AA), an extrapolation technique introduced in \cite{Anderson65}. AA has recently been used to improve convergence and robustness of nonlinear solvers for a wide range of problems including various types of flow problems \cite{LWWY12,PRX19,PRX21, LVX21}, 
	molecular interaction \cite{SM11}, 
	and many others e.g. \cite{WaNi11,K18,LW16,LWWY12,FZB20,WSB21}.  Hence applying it in this setting seems a natural next step. Indeed we show herein both theoretically and in numerical tests that AA-enhanced Picard maintains the Picard iteration's simplicity but provides it with much better efficiency and robustness, in particular for small $\eps$. {\color{black} For the sake of simplicity of the analysis, we consider homogeneous Dirichlet boundary condition. However, the extending the analysis to mixed Dirichlet/Neumann boundary problems is straightforward.
	}

	In addition to the study of nonlinear solvers, we will also consider the accuracy of a standard mixed finite element approximation of the regularized Bingham equations.  While some results exist in the literature
	for related variational inequality formulations \cite{DL76,G84} and particular low order stabilized elements \cite{LV04,FL18}, there seems to be not much done for general mixed finite element approximations.  
	 This may be due to the difficulty in solving the system resulting from 
		standard mixed methods as conventional nonlinear solvers will not converge for even 
		moderately small $\eps$ \cite{AHOV11}. 
		Here AA is seen to be an enabling technology
		as the solver now remains robust for small $\eps$.  
		Hence, for completeness, we include a spatial convergence analysis for mixed finite 
		elements applied to the regularized Bingham equations. We find the expected result that 
		optimal convergence can be obtained but is inversely dependent on $\eps$, but we also 
		find that suboptimal convergence (by one order) can be obtained that is independent of 
		$\eps$.  With very small $\eps$, it is the latter result that is expected in practice and in our numerical tests we do not see any significant negative scaling with $\eps$.

	
	This paper is arranged as follows: Section 2 provides notation and mathematical preliminaries on the finite element discretization and AA. Section 3 presents the Picard iteration to solve the regularized Bingham equations and proves properties of the associated fixed point solution operator. Then, we give 
{\color{black}an acceleration result for}
AA applied to a Picard iteration. In section 4, we provide the results of several numerical tests, which demonstrate a significant positive impact of AA on the convergence. Finally, we provide convergence analysis the finite element discretization of the regularized Bingham equations, to support the numerical results in section 4 which indicate no negative scaling with $\eps$.
	
	\section{Mathematical Preliminaries}\label{prelimsection}
	We consider a domain $\Omega\subset\mathbb{R}^d\ (d=2,\ 3)$ which is polygonal for $d=2$
	or polyhedral for $d=3$ (or $\partial\Omega\in C^{0,1}$). 
	The notation $\|\cdot\|$ and $(\cdot,\cdot)$ will be used to denote the $ L^2(\Omega) $ norm and inner product. The $H^{k}(\Omega)$ seminorm will be denoted by $|\cdot|_{k}$. We use boldface letters for vector-valued functions.
	
	The natural velocity and pressure spaces for the Bingham equations are given by
	\begin{align*}
		\bX&:=(H_{0}^{1}(\Omega))^d=\{\bv\in L^2(\Omega)^d:\nabla\bv\in L^2(\Omega)^{d\times d}\ \textnormal{and}\ \bv=\textbf{0}\ \textnormal{on}\ \partial\Omega\},\nonumber\\
		Q&:=L^{2}_{0}(\Omega)=\{q\in L^2(\Omega):\int_{\Omega} q\ d\bx=0 \}.\nonumber
	\end{align*}
	The Poincar\'e-Friedrichs' inequality is known to hold in $\bX$: For every $\bu\in\bX$,
	\begin{align*}
		\|\bv\|\leq C_F \|\nabla\bv\|,
	\end{align*} 
	where $C_F$ constant depending on the size of $\Omega$. Also, we define the divergence-free vector function space by
	\begin{align*}
		\bV:=\{ \bv\in \bX: (\nabla\cdot\bv,q)=0\ \forall q\in Q \}.
	\end{align*}
	From the vector identities $2\textbf{div} D=\Delta+\nabla\nabla\cdot$ and $\nabla\nabla\cdot=\Delta+\nabla\times\nabla\times\nabla$ applying integration by parts one gets the following Korn type inequalities
	\begin{align*}
		\|\nabla\bv\|\leq C_k \|D\bv\|,
	\end{align*}
	for all $\bv\in\bX.$
	
	The weak formulation of (\ref{regeqn}) can be written as follows: find $\bu\in \bX$ and $p\in Q$ such that
	\begin{equation}
		\begin{aligned}\label{weak}
			2\mu ( D\bu, D\bv ) 
			+\tau_s
			\left(\frac{D\bu}{|D\bu|_{\varepsilon}},D\bv\right)
			-( p, \nabla\cdot\bv)
			&=
			(\blf,\bv),\\
			(q, \nabla\cdot\bu) &=0.
		\end{aligned}
	\end{equation}
	Existence and the uniqueness of solutions can be proven by the Browder–Minty method of strictly monotone operators \cite{BS08},  for any $\eps>0$ and {\color{black} $f\in H^{-1}(\Omega)$} \cite{AHOV11}.
{	\color{black}
\begin{remark}
While the well-posedness of the system holds for any fixed $\eps>0$, as $\eps$ goes to zero, the bounds used for regularity and uniqueness blow up \cite{AHOV11} and there is 
 no rigorous study to extend the results in \cite{AHOV11} to the limit case of $\varepsilon=0$.  Still, the well-posedness of the unregularized system holds in 2D \cite{DL76,FN05} and in 3D existence is known but uniqueness is seemingly an open problem \cite{DL76}, these results are proved with different techniques, which suggests a potential gap in the known analysis for Bingham.
\end{remark}
}
	\subsection{Discretization Preliminaries}
	For the discrete setting, we assume a regular conforming triangulation $\tau_h(\Omega)$ with maximum element diameter $h$. Let $(\bX_h,Q_h)\subset (\bX,Q)$ be pair of discrete velocity-pressure spaces satisfying the LBB condition: there exists a constant $\beta$, independent of $h$ satisfying
	\begin{align}\label{discinfsup}
		\inf_{q\in Q_h} \sup_{\bv\in \bX_h}\frac{(\nabla \cdot\bv_h, q_h)}{\|q_h\| \|\nabla\bv_h\|}\geq \beta >0.
	\end{align}
	For simplicity, we assume $\bX_h=\bX\bigcap P_s(\tau_h)$ and $Q_h=Q\bigcap P_{r}(\tau_h)$, however, the analysis that follows can be applied to any inf-sup stable pair with only minor modifications.
	
	The space for discrete divergence free functions is
	$$\bV_h:=\{ \bv\in \bX_h: (\nabla\cdot\bv_h,q_h)=0\ \forall q_h\in Q_h \}.$$
	We assume the mesh is sufficiently regular for the inverse inequality to hold: there exists a constant $ C $ such that for all $\bv_h\in\bX_h$,
	\begin{align}\label{invineq}
		\|\nabla \bv_h\|\leq C h^{-1} \|\bv_h\|,
	\end{align}
	and with this and the LBB assumption, we assume interpolation operator $I_h: H^1(\Omega)\rightarrow \bV_h$ satisfying for all $\bv\in \bV$,
	\begin{align*}
		\|\bv-I_h(\bv)\| & \leq C h^{s+1} |\bv|_{s+1}, \\
		\|\nabla(\bv-I_h(\bv))\| & \leq C h^{s}|\bv|_{s+1}.
	\end{align*}
	We recall the following discrete Sobolve inequality in $\Omega\subset\mathbb{R}^2 $ (see \cite{BS08}),
	\begin{align}\label{discsobolevineq}
		\|\bv_h\|_{L^{\infty}(\Omega)}\leq C(1+|\ln h|)^{1/2} \|\nabla \bv_h\|\ \forall\bv\in \bX_h,
	\end{align}
	and for $\Omega\subset\mathbb{R}^3 $ and a quasi-uniform triangulation of $\Omega$, it follows from an Agmon's inquality and a standard inverse estimate \cite{BS08}
	that
	\begin{align}\label{discsobolevineq3D}
		\|\bv_h\|_{L^{\infty}(\Omega)}\leq C h^{-1/2} \|\nabla \bv_h\|\ \forall\bv\in \bX_h,
	\end{align}
	where $C$ is positive constant and independent of $h$.
	\subsection{Finite element Discretization of regularized Bingham equations}
	In this section, we present a FEM scheme for regularized Bingham equations (\ref{regeqn}). First, we define the FEM scheme as follows: Find  $(\bu_h,q)\in (\bX_h,Q_h)$ such that
	\begin{equation}
		\begin{aligned}\label{fem}
			2\mu ( D\bu_h, D\bv_h ) 
			+\tau_s\left(\frac{D\bu_h}{|D\bu_h|_{\varepsilon}},D\bv_h\right)
			-( p_h, \nabla\cdot\bv_h)
			&=
			(\blf,\bv_h),\\
			(q_h, \nabla\cdot\bu_h) &=0,
		\end{aligned}
	\end{equation}
	for all $ (\bv_h,q_h)\in(\bX_h,Q_h)$.
	
	The scheme (\ref{fem}) restricted to discretely divergence-free function space $\bV_h$ for velocity reads: Find  $\bu_h\in \bV_h$ such that for all  $\bv_h\in \bV_h,$
	\begin{align}\label{discdivfreeprob}
		a_{\eps}(\bu_h,\bv_h)
		=
		2\mu ( D\bu_h, D\bv_h ) 
		+\taus\left(\frac{D\bu_h}{|D\bu_h|_{\varepsilon}},D\bv_h\right)
		=
		(\blf,\bv_h).
	\end{align}
	We note due to the assumed LBB condition that \eqref{fem} and \eqref{discdivfreeprob} are equivalent.
	
	The well-posedness of scheme (\ref{discdivfreeprob}) follows the same as the well-posedness proof in \cite{AHOV11} for the analogous variational formulation posed in $\bV$ instead of $\bV_h$. The key steps rely on monotonicity which can be shown as follows: 
	\begin{equation}\label{monotonicityfem}
		\begin{aligned}
			a_{\varepsilon}&(\bu_h,\bu_h-\bv_h)-a_{\varepsilon}(\bv_h,\bu_h-\bv_h)
			\\
			&=
			\int_{\Omega} 2\mu |D\bu_h-D\bv_h|^2 + \taus \left( \frac{D\bu_h -D\bv_h}{|D\bu_h|_{\eps}} + \left(\frac{1}{|D\bu_h|_{\eps}}  -\frac{1}{|D\bv_h|_{\eps}}\right) D\bv_h  \right) : (D\bu_h - D\bv_h)
			\\
			&=
			\int_{\Omega} 2\mu |D\bu_h-D\bv_h|^2 + \frac{\taus}{|D\bu_h|_{\eps}} \left( |D\bu_h-D\bv_h|^2 - \frac{|D\bu_h|_{\eps}-|D\bv_h|_{\eps}}{|D\bv_h|_{\eps}} D\bv_h : (D\bu_h - D\bv_h) \right) 
			\\
			&\geq
			\int_{\Omega} 2\mu |D\bu_h-D\bv_h|^2 + \frac{\taus}{|D\bu_h|_{\eps}} \left( |D\bu_h-D\bv_h|^2 - \frac{|D\bu_h-D\bv_h|}{|D\bv_h|_{\eps}} D\bv_h : (D\bu_h - D\bv_h) \right) 
			\\
			&\geq
			2\mu \|D\bu_h-D\bv_h\|^2,
		\end{aligned}
	\end{equation}	
	where $||D\bv_h|^{-1}_{\eps}D\bv_h|\leq 1$. The Browder-Minty theorem then guarantees
	existence and uniqueness of the solution.
	\subsection{Anderson acceleration}\label{aasection}
	Anderson acceleration (AA) is an extrapolation technique that is used to improve convergence of fixed-point iterations. Consider a fixed-point operator $g:Y\rightarrow Y$ where Y is a normed vector space. The AA procedure is stated in the following algorithm: Denote
	$w_{j} = g(x_{j-1}) - x_{j-1}$ as the nonlinear residual, 
	also sometimes referred to as the update step.
	\begin{algorithm} \label{alg:anderson}
		(Anderson acceleration with depth $m$ and damping factors $\beta_k$)\\ 
		Step 0: Choose $x_0\in Y.$\\
		Step 1: Find $w_1\in Y $ such that $w_1 = g(x_0) - x_0$ {\color{black}where $g(x_0)=x_1$.}
		Set $x_1 = x_0 + w_1$. \\
		Step $k$: For $k=2,3,\ldots$ Set $m_k = \min\{ k-1, m\}.$\\
		\indent [a.] Find $w_{k} = g(x_{k-1})-x_{k-1}$. \\
		\indent [b.] Solve the minimization problem for the Anderson coefficients $\{ \alpha_{j}^{k}\}_{k-m_k}^{k-1}$
		\begin{align}\label{eqn:opt-v0}
			\{ \alpha_{j}^{k}\}_{k-m_k}^{k-1}=\textstyle \text{argmin} 
			\left\| \left(1- \sum\limits_{j=k-m_k}^{k-1} \alpha_j^{k} \right) w_k + \sum\limits_{j = k-m_k}^{k-1}  \alpha_j^{k}  w_{j} \right\|_Y.
		\end{align}
		\indent [c.] For damping factor $0 < \beta_k \le 1$, set
		\begin{align}\label{eqn:update-v0}
			\textstyle
			x_{k} 
			= (1-\sum\limits_{j = k-m_k}^{k-1}\alpha_j^k) x_{k-1} + \sum_{j= k-m_k}^{k-1} \alpha_j^{k} x_{j-1}
			+ \beta_k \left(  (1- \sum\limits_{j= k-m_k}^{k-1} \alpha_j^{k}) w_k + \sum\limits_{j=k-m_k}^{k-1}\alpha_j^k w_{j}\right).
		\end{align}
	\end{algorithm}
	The $m=0$ case is equivalent to the fixed point iteration without acceleration. To understand how AA improves convergence, define matrices $ E_k $ and $ F_k $, whose columns are the consecutive differences between iterates and residuals, respectively.
	\begin{align}
		E_k&:=\left( e_{k-1}\ e_{k-2}\ \dotsc\ e_{k-m_k}   \right), e_j=x_j-x_{j-1} \label{Fk1}\\
		F_k&:=\left( (w_{k}-w_{k-1})(w_{k-1}-w_{k-2})\ \dotsc\ (w_{k-m_k+1}-w_{k-m_k})  \right).\label{Fk2}
	\end{align}
	Then defining $\gamma^k=\argmin_{\gamma\in\mathbb{R}^m}\| w_k -F_k\gamma \|_Y$, the update step \eqref{eqn:update-v0} can be written as
	\begin{align*}
		x_k=x_{k-1} - \beta_k w_k - (E_k + \beta_k F_k) \gamma^k = x_{k-1}^{\alpha} + \beta_k w_k^{\alpha},
	\end{align*}
	where $w_k^{\alpha}=w_k - F_k \gamma^k$ and $x^{\alpha}_{k-1}=x_{k-1}-E_{k-1} \gamma^k$ are the averages corresponding to the solution from the optimization problem. The optimization gain factor $\theta_k$ may be defined by
	\begin{align*}
		\|w_k^{\alpha}\|=\theta_k \|w_k\|.
	\end{align*}
	As shown in the recent theory proposed in \cite{PR21, PRX19}, the gain factor $\theta_k$ is the key to acceleration.
	
	The next two assumptions from \cite{PR21} provide sufficient conditions on the fixed point operator $g$ for the 
acceleration results developed
	therein to hold.
	\begin{assumption}\label{assume:g}
		Assume $g\in C^1(Y)$ has a fixed point $x^\ast$ in $Y$, and there are positive constants $C_0$ and $C_1$ with
		\begin{enumerate}
			\item $\|g'(x)\|_{Y}\le C_0$ for all $x\in Y$, and 
			\item $\|g'(x) - g'(y)\|_{Y} \le C_1 \|x-y\|_{Y}$
			for all $x,y \in Y$.
		\end{enumerate}
	\end{assumption}
	\begin{assumption}\label{assume:fg} 
		Assume there is a constant $\sigma> 0$ for which the differences between consecutive
		residuals and iterates satisfy
		\begin{align} \label{eqn:assumefg}
			\| w_{{k}+1} - w_{k}\|_Y  \ge \sigma \| x_{k} - x_{{k}-1} \|_Y, \quad {k} \ge 1.
		\end{align}
	\end{assumption}
	Assumption \ref{assume:g} describes properties of the underlying fixed-point operator.
	Both parts of this assumption will {\color{black}be verified} for the Picard fixed-point operator
	for this problem in the analysis that follows.
	Assumption \ref{assume:fg} is harder to verify for this problem. It is globally 
	satisfied for instance if $g$ is contractive, which is generally not the case here, or
	locally if the Jacobian of $g$ can be shown not to degenerate in the vicinity of a
	solution, as discussed in \cite{PR21}. 
	On the other hand, \eqref{eqn:assumefg} can be checked at each iteration as it only 
	involves the differences between iterates and update steps that have already been 
	computed. 
	Assumption \ref{assume:fg} 
	can then be enforced for instance by the following safeguarding 
	strategy: given some chosen $\bar \sigma>0$, 
	on any step for which \eqref{eqn:assumefg} is not satisfied with $\sigma = \bar \sigma$,
	the next iterate can be given by the simple fixed-point iteration, after 
	which AA can be restarted. 
	We found it was not necessary to implement this strategy here, however.
	We demonstrate this in section 4, where we calculate the ratio 
	$\| w_{k+1}-w_k\|_Y / \|x_k -x_{k-1}\|_X$ 
	using varying fixed $m$ on a benchmark problems and find $\sigma$ bounded well above 0.

	Under Assumptions \ref{assume:g} and \ref{assume:fg}, the following result summarized from \cite{PR21}, produces a one-step bound on the residual $\|w_{k+1}\|$ in terms of the previous residual $\|w_k\|$.
	\begin{theorem}[Pollock et al., 2021]  \label{thm:genm}
		Let Assumptions \ref{assume:g} and \ref{assume:fg} hold,
		and suppose the direction sines between each column $i$ of $F_j$ defined by 
		\eqref{Fk1} and the subspace spanned by the preceeding columns satisfy
		$|\sin(f_{j,i},\text{span }\{f_{j,1}, \ldots, f_{j,i-1}\})| \ge c_s >0$,
		for $j = k-m_k, \ldots, k-1$.
		Then the residual $w_{k+1} = g(x_k)-x_k$ from Algorithm \ref{alg:anderson} 
		(depth $m$) satisfies the following bound.
		\begin{align}\label{eqn:genm}
			\nr{w_{k+1}} & \le \nr{w_k} \Bigg(
			\theta_k ((1-\beta_{k}) + C_0 \beta_{k})
			+ \frac{C C_1\sqrt{1-\theta_k^2}}{2}\bigg(
			\nr{w_{k}}h(\theta_{k})
			\nonumber \\ &
			+ 2  \sum_{n = k-{m_{k}}+1}^{k-1} (k-n)\nr{w_n}h(\theta_n) 
			+ m_{k}\nr{w_{k-m_{k}}}h(\theta_{k-m_{k}})
			\bigg) \Bigg),
		\end{align}
		where  each $h(\theta_j) \le C \sqrt{1 - \theta_j^2} + \beta_{j}\theta_j$,
		and $C$ depends on $c_s$ and the implied upper bound on the direction cosines. 
	\end{theorem}
	In this estimate, $\theta_k$ is the gain from the optimization problem, and it determines the relative scalings of the contributions from the lower and higher order terms. 
The lower order terms are multiplied by $\theta_k$, and the higher-order terms are 
multiplied by $\sqrt{1-\theta_k^2}.$ 
{\color{black} While this bound does not guarantee global convergence, it does establish
how AA improves the first order term at the cost of adding higher-order terms to the
residual expansion at each step. For contractive problems, this additionally shows
local convergence with an improved rate in comparison to the original fixed-point 
iteration. If close to the root (so higher order terms are negligible), it shows AA 
will improve the the convergence rate by the scaling factor $\theta_k$, which can change 
at each step.}  
	
	\section{Acceleration of the regularized Bingham Picard iteration}\label{picardsection}
	In this section, we present some properties of the Picard iteration to solve (\ref{regeqn}) and its associated fixed point function. 
	We then use these properties to apply convergence and acceleration theory
	for AA to this iteration.
	Given $\bu^0\in\bV_h$, for $k=1,2,...,$ find $\bu^{k}\in\bV_h$ such that
	\begin{equation}\label{picardbinghamV}
		\begin{aligned}
			2\mu (D \bu^{k}, D\bv ) 
			+
			\tau_s
			\left(
			\frac{D\bu^{k}}{|D\bu^{k-1}|_{\varepsilon}},D\bv
			\right)
			&=
			(\blf,\bv),\quad \forall\bv\in\bV_h.
		\end{aligned}
	\end{equation}
	
	The convergence analysis of this Picard iteration for  (\ref{regeqn}) is given in \cite{AHOV11}.
	\subsection{Solution operator $G$ corresponding to the Picard iteration}
	In this subsection, we study some properties of the solution operator of the linearized problem of the form of (\ref{picardbinghamV}).
	
	Let $\blf\in \bH^{-1}(\Omega)$ and $\bu\in\bV_h$ be given. Consider the problem of finding $\tilde{\bu}\in\bV_h$ such that
	\begin{equation}\label{picardbinghamG}
		\begin{aligned}
			2\mu ( D\tilde{\bu}, D\bv ) 
			+
			\tau_s
			\left(
			\frac{D\tilde{\bu}}{|D\bu|_{\varepsilon}},D\bv
			\right)
			&=
			(\blf,\bv),\quad \forall\bv\in\bV_h.
		\end{aligned}
	\end{equation}
	In continuous case, well-posedness and convergence analysis of the solution of Picard iteration of (\ref{regeqn}) is presented in \cite{AHOV11}. In discrete setting, well-posedness can be proven by following these same steps.
	\begin{lemma}\label{welltildeu}
		For  $\blf\in \bH^{-1}(\Omega)$ and $\bu\in\bV_h$, (\ref{picardbinghamG}) is well-posed and the solution satisfies the bound
		\begin{align}\label{apriorib}
			\| \nabla\tilde{\bu} \|
			\leq
			\mu^{-1}\|\blf\|_{-1}.
		\end{align}
	\end{lemma}
	\begin{proof}
		Assume a solution exists, and choose $\bv=\tilde{\bu}\in\bV_h$. Then, using the dual norm on $\bV_h$, we get
		\begin{align*}
			\mu \| \nabla\tilde{\bu} \|^2
			\leq
			2\mu \| D\tilde{\bu} \|^2
			\leq
			2\mu \| D\tilde{\bu} \|^2
			+\tau_s\| |D\bu|_{\varepsilon}^{-1/2}  D\tilde{\bu}\|^2
			=
			(\blf,\tilde{\bu})
			\leq
			\|\blf\|_{-1} \|\nabla\tilde{\bu}\|,
		\end{align*}
		which shows (\ref{apriorib}). This bound is 
		sufficient to imply uniqueness since the system is linear, and since it is also finite dimensional, existence follows from uniqueness.
	\end{proof}
	\begin{definition}\label{defineG}
		Define $G:\bV_h\rightarrow \bV_h$ to be the solution operator of (\ref{picardbinghamG}). That is,
		$$\tilde{\bu}=G(\bu).$$
	\end{definition}
	By Lemma \ref{welltildeu}, (\ref{picardbinghamG}) is well-posed, so $G$ is well defined. Thus, the iteration (\ref{picardbinghamV}) can now be written as
	$$\bu^{k+1}=G(\bu^{k}).$$
	\subsection{Lipschitz continuity and differentiability of G}
	In this subsection, we prove properties of $G$ which are used to show convergence of the AA Picard iteration for (\ref{discdivfreeprob}) via Theorem \ref{thm:genm}.
	First, to prove that $ G $ satisfies the first part of Assumptions \ref{assume:g}, we show that $G$ is Lipschitz continuous, $G'$ exists and is the Fr\'echet derivative of $G$. Then, by showing $G$ is Lipschitz continuously differentiable, we prove that $ G $ satisfies the second part of Assumptions \ref{assume:g}. 
	The satisfaction of both properties allows us to establish convergence of the AA Picard 
	iteration for (\ref{picardbinghamV}).
	We begin with Lipschitz continuity of $G$.
	\begin{lemma}\label{LipschitzG}
		For any $\bu, \bw\in \bV_h$, we have
		\begin{align}\label{CG}
			\| DG(\bu)-DG(\bw)\|
			\leq
			C_G
			\|D\bw-D\bu\|,
		\end{align}	
		where $C_G=\left(
		\frac{\tau_s
			\eps^{-3}
			C(1+|\ln h|) h^{-2} \mu^{-3} \|\blf\|_{-1}^2 }{8}
		\right)^{1/2}$ in 2D, and $C_G=	\left(
		\frac{\tau_s
			\eps^{-3}
			C h^{-3} \mu^{-3} \|\blf\|_{-1}^2 }{8}\right)^{1/2}$ in 3D.
	\end{lemma}
	{\color{black}
	\begin{remark}\label{rem1}
	This constant is quite large, but it holds globally and also we make no assumptions on the data (i.e. no assumption that $\mu$ is large).  In terms of representing a contraction number, we believe this to be a pessimistic bound.  While negative scalings with $h$ and $\epsilon$ are observed in our tests and in \cite{AHOV11}, the negative scalings 
	are much milder than these.  It appears to be an open problem to show the existence of a region (i.e. $u$ close enough to $w$) where $C_G<1$, without excessive restrictions on the data and mesh size.
	\end{remark}
	}

	\begin{proof}
		Let $\bu,\bw\in \bV_h$ and $G(\bu)=\tilde{\bu}$ and $G(\bw)=\tilde{\bw}$. Then,
		\begin{equation}\label{upsoln}
			\begin{aligned}
				2\mu ( DG(\bu), D\bv ) 
				+
				\tau_s
				\left(
				\frac{DG(\bu)}{|D\bu|_{\varepsilon}},D\bv
				\right)
				&=
				(\blf,\bv).
			\end{aligned}
		\end{equation}
		\begin{equation}\label{wzsoln}
			\begin{aligned}
				2\mu ( DG(\bw), D\bv ) 
				+
				\tau_s
				\left(
				\frac{DG(\bw)}{|D\bw|_{\varepsilon}},D\bv
				\right)
				&=
				(\blf,\bv).
			\end{aligned}
		\end{equation}
		Subtracting (\ref{wzsoln}) from (\ref{upsoln}), then adding and subtracting $\frac{DG(\bw)}{|D\bu|_{\varepsilon}}$ from the first argument in the second term, we get
		\begin{equation}
			\begin{aligned}
				2\mu ( DG(\bu)-DG(\bw), D\bv ) 
				+
				\taus
				\left(
				\frac{1}{|D\bu|_{\varepsilon}}
				\left(
				DG(\bu)-DG(\bw)
				\right)
				,D\bv
				\right)
				\\+
				\taus
				\left(
				\left(
				\frac{1}{|D\bu|_{\varepsilon}}
				-\frac{1}{|D\bw|_{\varepsilon}}
				\right)
				DG(\bw),D\bv
				\right)
				=0.
			\end{aligned}
		\end{equation}
		Choosing $\bv= G(\bu)-G(\bw)$ gives
		\begin{equation}
			\begin{aligned}
				2\mu \| DG(\bu)-DG(\bw)\|^2 
				+
				\tau_s
				\|\
				|D\bu|_{\varepsilon}^{-1/2}
				(
				DG(\bu)-DG(\bw))
				\|^2
				\\
				=
				-
				\tau_s
				\left(
				\left(
				\frac{1}{|D\bu|_{\varepsilon}}
				-\frac{1}{|D\bw|_{\varepsilon}}
				\right)
				DG(\bw),DG(\bu)-DG(\bw)
				\right),
			\end{aligned}
		\end{equation}
		and then using reverse triangle and H\"older's inequalities, noting that $\||D\bu|_{\eps}^{-1}\|_{L^{\infty}(\Omega)}\leq \eps^{-1}$, exploiting discrete Sobolev and inverse inequalities, (\ref{apriorib}) and Young's inequality, we obtain in 2D that
		\begin{align*}
			&\left|
			-
			\tau_s
			\left(
			\left(
			\frac{1}{|D\bu|_{\varepsilon}}
			-\frac{1}{|D\bw|_{\varepsilon}}
			\right)
			DG(\bw),DG(\bu)-DG(\bw)
			\right)
			\right|
			\\
			&\,\,\,\,\,\,\leq
			\tau_s
			\int_{\Omega}
			\frac{\left|D\bw-D\bu\right|}{|D\bu|_{\varepsilon}|D\bw|_{\varepsilon}}
			|DG(\bw)|\ 
			|DG(\bu)-DG(\bw)|
			\\
			&\,\,\,\,\,\,\leq
			\tau_s
			\eps^{-3/2}
			C(1+|\ln h|)^{1/2} h^{-1} \mu^{-1} \|\blf\|_{-1}
			\|D\bw-D\bu\|\ 
			\|\ |D\bu|_{\varepsilon}^{-1/2} (DG(\bu)-DG(\bw))\|
			\\
			&\,\,\,\,\,\,\leq
			\frac{\tau_s
				\eps^{-3}
				C(1+|\ln h|) h^{-2} \mu^{-2} \|\blf\|_{-1}^2 }{4}
			\|D\bw-D\bu\|^2
			+
			\taus 
			\|\ |D\bu|_{\varepsilon}^{-1/2} (DG(\bu)-DG(\bw))\|^2.
		\end{align*}
		So, combining the bound for left hand side term and dividing each side by $2\mu$ give
		\begin{align*}
			\| DG(\bu)-DG(\bw)\|^2 
			\leq
			\frac{\tau_s
				\eps^{-3}
				C(1+|\ln h|) h^{-2} \mu^{-3} \|\blf\|_{-1}^2 }{8}
			\|D\bw-D\bu\|^2.
		\end{align*}
		Then, by taking the square roots of both sides, we get (\ref{CG}). For the 3D case, we use inverse inequality (\ref{discsobolevineq3D}) instead of \eqref{discsobolevineq} to obtain the result.
	\end{proof}
	
	Next, we show that G is Lipschitz Fr\'echet differentiable. We begin by defining the operator $G',$ and then show it is the Fr\'echet derivative operator of G.
	\begin{definition}\label{defineG'}
		Given $\bu\in\bV_h$, define an operator $G'(\bu;\cdot):\bV_h\rightarrow\bV_h$ by $G'(\bu;\bh)$ satisfying for all $\bh\in\bV_h$,
		\begin{equation}\label{diffG}
			\begin{aligned}
				2\mu ( DG'(\bu;\bh), D\bv ) 
				+\tau_s\left(
				\frac{DG'(\bu;\bh)}{|D\bu|_{\eps}},D\bv
				\right)
				=\tau_s\left(
				\frac{D\bu:D\bh}{|D\bu|^3_{\eps}}DG(\bu),D\bv
				\right).
			\end{aligned}
		\end{equation}
	\end{definition}
	
	Now, we need to show $G'$ is the Jacobian matrix of $G$ at $\bu$ by the following lemma. To show this, first we need to prove $G'$ in definition (\ref{defineG'}) is well defined.
	\begin{lemma} 
		\label{lemma:wellG'}
		The operator $G'$ in Definition (\ref{defineG'}) is well-defined for all $\bu,\bh\in\bV_h$ such that
		\begin{align}\label{wellG'}
			\| DG'(\bu;\bh)\|
			\leq
			C_G \|D\bh\|,
		\end{align}
		where $C_G=C_{G_{2D}}$ on $\Omega\subset\mathbb{R}^2$, and $C_G=C_{G_{3D}}$ on $\Omega\subset\mathbb{R}^3$.
	\end{lemma}
	\begin{proof}
		The proof of Lemma \ref{lemma:wellG'} can be done by following same steps 
		as in
		the proof of Lemma \ref{LipschitzG}. Since (\ref{diffG}) is linear and finite dimensional, (\ref{wellG'}) is sufficient to say that the (\ref{diffG}) is well-posed. Thus, $G'$ is well-defined and uniformly bounded over $\bV_h$, since the bound is independent of $\bu$.
	\end{proof}
	Next, we show that $G'$ is the Fr\'echet derivative operator of $G$. That is, given $\bu\in\bV_h$, there exists some constant $\mathcal{F}$ such that for any $\bh\in\bV_h$
	\begin{align*}
		\|D(G(\bu+\bh)-G(\bu)-G'(\bu;\bh))\| \leq \mathcal{F}\|D\bh\|^2.
	\end{align*}

	\begin{lemma}\label{lemma:frechetGlemma}
		For arbitrary $\bu\in \bV_h$ and sufficiently small $\bh\in \bV_h$, the bound
		\begin{align}\label{frechetGlemma}
			\| DG(\bu+\bh)-DG(\bu)-DG'(\bu;\bh)\|
			\leq
			\left(
			\taus
			C(1+|\ln h|) h^{-2} \eps^{-3}
			\left(
			C_G^2
			+
			\eps^{-2}
			\mu^{-2} \|\blf\|^2_{-1}
			\right)\right)^{1/2}
			\|D\bh\|^2
		\end{align}
		holds, which implies $G$ is Fr\'echet differentiable on $\bV_h$.
	\end{lemma}

	\begin{proof}
		Set  $\tilde{\bg}= G(\bu+\bh)-G(\bu)-G'(\bu;\bh)$ for notational ease. To construct the left hand side of the inequality above, we begin with the following equations: for any $\bu,\bh\in\bV_h$,
		\begin{equation}\label{uhsoln}
			\begin{aligned}
				2\mu ( DG(\bu+\bh), D\bv ) 
				+\taus\left(
				\frac{DG(\bu+\bh)}{|D(\bu+\bh)|_{\eps}},D\bv
				\right)
				&=
				(\blf,\bv).
			\end{aligned}
		\end{equation}
		Subtracting (\ref{upsoln}) and (\ref{diffG}) from (\ref{uhsoln}), we obtain
		\begin{align}\label{eqn:2term}
			2\mu ( D\tilde{\bg}, D\bv ) 
			+\taus\left(
			\frac{DG(\bu+\bh)}{|D(\bu+\bh)|_{\eps}}
			-
			\frac{DG(\bu)}{|D\bu|_{\varepsilon}}
			-
			\frac{DG'(\bu;\bh)}{|D\bu|_{\eps}}
			-
			\frac{D\bu:D\bh}{|D\bu|^3_{\eps}}DG(\bu)
			,D\bv
			\right)
			=
			0.
		\end{align}
		Adding and subtracting $\frac{DG(\bu+\bh)}{|D\bu|_{\eps}}$ from the first argument in the second term on the left hand side of \eqref{eqn:2term}
		and then choosing $\bv=\tilde{\bg}$ gives
		\begin{align*}
			2\mu \| D\tilde{\bg}\|^2 
			+\taus
			\||D(\bu)|_{\eps}^{-1/2} D\tilde{\bg}\|^2
			=
			-
			\taus\left(
			\frac{DG(\bu+\bh)}{|D(\bu+\bh)|_{\eps}}
			-
			\frac{DG(\bu+\bh)}{|D\bu|_{\eps}}
			-
			\frac{D\bu:D\bh}{|D\bu|^3_{\eps}}DG(\bu)
			,
			D\tilde{\bg}
			\right).
		\end{align*}
		Adding and subtracting $\frac{D\bu:D\bh}{|D\bu|^3_{\eps}}DG(\bu+\bh)$ from the first argument of the term on right hand side of \eqref{eqn:2term} and rearranging terms gives
		\begin{align}
			2\mu \| D\tilde{\bg}\|^2 
			+\taus
			\||D(\bu)|_{\eps}^{-1/2} D\tilde{\bg}\|^2
			=
			&-
			\taus\left(
			\frac{D\bu:D\bh}{|D\bu|^3_{\eps}}
			\left(
			DG(\bu+\bh)
			-
			DG(\bu)
			\right)
			,
			D\tilde{\bg}
			\right) \nonumber
			\\&-
			\taus\left(
			\left(
			\frac{1}{|D(\bu+\bh)|_{\eps}}
			-
			\frac{1}{|D\bu|_{\eps}}
			-
			\frac{D\bu:D\bh}{|D\bu|^3_{\eps}}
			\right)
			DG(\bu+\bh)
			,
			D\tilde{\bg}
			\right). \label{GFD}
		\end{align}
		
		We now estimate the right hand side terms of \eqref{GFD}.  For the first one, we use Lemma \ref{LipschitzG}, that  $\| |D\bu|^{-1}_{\eps}\|_{L^{\infty}(\Omega)}\leq\eps^{-1}$ and $\||D\bu|_{\eps}^{-1} D(\bu)\|_{L^{\infty}(\Omega)}\leq 1$, and H\"older's, discrete Sobolev, inverse and Young's inequalities to obtain
		\begin{align*}
			&\left|
			-
			\taus\left(
			\frac{D\bu:D\bh}{|D\bu|^3_{\eps}}
			\left(
			DG(\bu+\bh)
			-
			DG(\bu)
			\right)
			,
			D\tilde{\bg}
			\right)
			\right|
			\\&\leq
			\taus
			\int_{\Omega} 
			\frac{|D\bu| |D\bh|}{|D\bu|^3_{\eps}}
			\left|
			DG(\bu+\bh)
			-
			DG(\bu)
			\right|
			|D\tilde{\bg}|
			\\
			&\leq
			\taus
			\eps^{-3/2}
			\| DG(\bu+\bh)-DG(\bu)\|_{L^{\infty}(\Omega)}
			\|D\bh\|
			||D\bu|^{-1/2}_{\eps}D\tilde{\bg}|
			\\
			&\leq
			\taus \eps^{-3/2}
			C(1+|\ln h|)^{1/2} h^{-1}
			\|DG(\bu+\bh)-DG(\bu)\|
			\|D\bh\|
			||D\bu|^{-1/2}_{\eps}D\tilde{\bg}|
			\\
			&\leq
			\taus \eps^{-3/2}
			C(1+|\ln h|)^{1/2} h^{-1}
			C_G
			\|D\bh\|^2
			||D\bu|^{-1/2}_{\eps}D\tilde{\bg}\|
			\\
			&\leq
			\frac{\taus}{2} \eps^{-3}
			C(1+|\ln h|) h^{-2}
			C_G^2
			\|D\bh\|^4
			+
			\frac{\taus}{2}
			||D\bu|^{-1/2}_{\eps}D\tilde{\bg}\|^2.
		\end{align*}
		
		For the second term in \eqref{GFD}, we proceed similar to the first term but utilize the Taylor expansion
		\begin{align*}
			\frac{1}{|D(\bu+\bh)|_{\eps}}
			=
			\frac{1}{|D\bu|_{\eps}}
			+
			\frac{D\bu:D\bh}{|D\bu|^3_{\eps}}
			+
			\frac{1}{2}
			\left(
			\frac{1}{|D\bu|^3_{\eps}}
			+
			3\frac{D\bu:D\bu}{|D\bu|^5_{\eps}}
			\right)
			|D\bh|^2
			+
			{\color{black}\textit{higher-order terms},}
		\end{align*}
		to get 
		\begin{align*}
			&\left|
			-
			\taus\left(
			\left(
			\frac{1}{|D(\bu+\bh)|_{\eps}}
			-
			\frac{1}{|D\bu|_{\eps}}
			-
			\frac{D\bu:D\bh}{|D\bu|^3_{\eps}}
			\right)
			DG(\bu+\bh)
			,
			D\tilde{\bg}
			\right)
			\right|
			\\&\leq
			\taus
			\int_{\Omega}
			\left|
			\frac{1}{|D(\bu+\bh)|_{\eps}}
			-
			\frac{1}{|D\bu|_{\eps}}
			-
			\frac{D\bu:D\bh}{|D\bu|^3_{\eps}}
			\right|
			|DG(\bu+\bh)|
			|D\tilde{\bg}|
			\\
			&\leq
			\taus
			\int_{\Omega}
			\frac{3}{4}
			\left|
			\frac{1}{|D\bu|_{\eps}^3}
			+
			\frac{3|D\bu|^2}{|D\bu|_{\eps}^5}
			\right|
			|D\bh|^2
			|DG(\bu+\bh)|
			|D\tilde{\bg}| + {\color{black}\textit{higher-order terms}} 
			\\
			&\leq
			2\eps^{-5/2}
			C(1+|\ln h|)^{1/2} h^{-1} \mu^{-1} \|\blf\|_{-1}
			\|D\bh\|^2
			\||D\bu|_{\eps}^{-1/2} D\tilde{\bg} \|+	{\color{black}\textit{higher-order terms}}
			\\
			&\leq
			2\taus\eps^{-5}
			C(1+|\ln h|) h^{-2} \mu^{-2} \|\blf\|^2_{-1}
			\|D\bh\|^4
			+
			\frac{\taus}{2}
			\||D\bu|_{\eps}^{-1/2} D\tilde{\bg} \|^2.
		\end{align*}
		In the third line in the above inequality string we account for higher order terms by increasing the $\frac12$ coefficient
		from the Taylor expansion to be $\frac34$, {\color{black} since the higher-order terms are higher order in {\bf h} which we can consider arbitrarily small in this context, while the mesh
		and $\epsilon$ are considered fixed.}
		
		Combining the bounds above, we obtain
		\begin{align*}
			\| D\tilde{\bg}\|^2 
			\leq
			\taus
			C(1+|\ln h|) h^{-2} \eps^{-3}
			\left(
			C_G^2
			+
			\eps^{-2}
			\mu^{-2} \|\blf\|^2_{-1}
			\right)
			\|D\bh\|^4.
		\end{align*}
		So, by taking the square roots of both sides and applying the definition of $\tilde{\bg}$, we get
		\begin{align}\label{frechetG}
			\| DG(\bu+\bh)-DG(\bu)-DG'(\bu;\bh)\|
			\leq
			\left(
			\taus
			C(1+|\ln h|) h^{-2} \eps^{-3}
			\left(
			C_G^2
			+
			\eps^{-2}
			\mu^{-2} \|\blf\|^2_{-1}
			\right)\right)^{1/2}
			\|D\bh\|^2,
		\end{align}
		which shows  Fr\'echet differentiability of $G$ at $\bu$. Since (\ref{frechetG}) holds for arbitrary $\bu$, $G$ is Fr\'echet differentiable on $\bV_h$. In the case of  $\Omega\subset\mathbb{R}^3$, we apply inverse inequality (\ref{discsobolevineq3D}) instead of \eqref{discsobolevineq}, and the rest of the steps are identical.
	\end{proof}
	We now show $G'$ is Lipschitz continuous over $\bV_h$.
	\begin{lemma}\label{lemma:LipschitzG'}
		$G$ is Lipschitz continuously differentiable on $\bV_h$, such that for all $\bu,\bs,\bh\in\bV_h$
		\begin{align*}
			\| D\left( G'(\bu+\bh;\bs)-G'(\bu;\bs) \right)\|
			\leq
			\hat{C}_G
			\|D\bs\|
			\|D\bh\|,
		\end{align*}	
		where there exist a constant 
		$\hat{C}_G
		=C
		\mu^{-1}
		\tau_s
		\eps^{-2}
		(1+|\ln h|)^{1/2} h^{-1}
		\left(
		C_G 
		+
		\eps^{-1}
		(1+|\ln h|)^{1/2} h^{-1} \mu^{-1} \|\blf\|_{-1}
		\right)$ in 2D
		and
		$\hat{C}_G
		=C
		\mu^{-1}
		\tau_s
		\eps^{-2}
		h^{-3/2}
		\left(
		C_G 
		+
		\eps^{-1}
		h^{-3/2} \mu^{-1} \|\blf\|_{-1}
		\right)$ in 3D, and $C_G$ is defined in Lemma (\ref{LipschitzG}).
	\end{lemma}
	\begin{proof}
		By the definition of $G'$, the following equations hold
		\begin{align}\label{G'}
			2\mu ( DG'(\bu;\bs), D\bv ) 
			+\tau_s\left(
			\frac{DG'(\bu;\bs)}{|D\bu|_{\eps}},D\bv
			\right)
			+\tau_s\left(
			\frac{D\bu:D\bs}{|D\bu|^3_{\eps}}DG(\bu),D\bv
			\right)
			&=
			0,
		\end{align}
		\begin{align}\label{G'+h}
			2\mu ( DG'(\bu+\bh;\bs), D\bv ) 
			+\tau_s\left(
			\frac{DG'(\bu+\bh;\bs)}{|D(\bu+\bh)|_{\eps}},D\bv
			\right)
			+\tau_s\left(
			\frac{D(\bu+\bh):D\bs}{|D(\bu+\bh)|^3_{\eps}}DG(\bu+\bh),D\bv
			\right)
			&=
			0,
		\end{align}
		for all $\bu,\bs,\bh,\bv\in\bV_h$. 
		
		Set $\be=G'(\bu+\bh;\bs)-G'(\bu;\bs),$ and then by subtracting (\ref{G'}) from (\ref{G'+h}), we get
		\begin{align*}
			2\mu ( D\be, D\bv ) 
			+\tau_s
			\left(
			\frac{DG'(\bu+\bh;\bs)}{|D(\bu+\bh)|_{\eps}}-\frac{DG'(\bu;\bs)}{|D(\bu)|_{\eps}},D\bv
			\right)
			\\+\tau_s\left(
			\frac{D(\bu+\bh):D\bs}{|D(\bu+\bh)|^3_{\eps}}DG(\bu+\bh)
			-
			\frac{D\bu:D\bs}{|D(\bu)|^3_{\eps}}DG(\bu)
			,D\bv
			\right)
			=
			0.
		\end{align*}
		By adding and subtracting $\frac{DG'(\bu;\bs)}{|D(\bu+\bh)|_{\eps}}$ from the first argument in second term on left hand side and choosing $\bu=\be$, we obtain
		\begin{align*}
			2\mu \| D\be\|^2
			&\leq
			2\mu \| D\be\|^2
			+\tau_s
			\| |D(\bu+\bh)|_{\eps}^{-1/2}  D\be\|^2
			\\
			&=
			-\tau_s\left(
			\frac{D(\bu+\bh):D\bs}{|D(\bu+\bh)|^3_{\eps}}DG(\bu+\bh)
			-
			\frac{D\bu:D\bs}{|D(\bu)|^3_{\eps}}DG(\bu)
			,D\be
			\right)
			\\&
			-
			\tau_s
			\left(
			\left(\frac{1}{|D(\bu+\bh)|_{\eps}}-\frac{1}{|D(\bu)|_{\eps}}\right)DG'(\bu;\bs),D\be
			\right).
		\end{align*}
		Noting that $||D\bu|_{\eps}D\bu|\leq1$ and $\||D\bu|_{\eps}^{-1}\|_{L^{\infty}(\Omega)}\leq \eps^{-1}$, applying H\"older's, discrete Sobolev and inverse inequalities, (\ref{apriorib}) and (\ref{CG}), we get
		\begin{align*}
			&\left|-\tau_s\left(
			\frac{D(\bu+\bh):D\bs}{|D(\bu+\bh)|^3_{\eps}}DG(\bu+\bh)
			-
			\frac{D\bu:D\bs}{|D(\bu)|^3_{\eps}}DG(\bu)
			,D\be
			\right)\right|
			\\
			&\leq
			\tau_s
			\int_{\Omega}
			\frac{D\bu:D\bs}{|D(\bu)|^3_{\eps}}DG(\bu+\bh)D\be
			+
			\frac{D\bh:D\bs}{|D(\bu+\bh)|^3_{\eps}}DG(\bu+\bh)D\be
			-
			\frac{D\bu:D\bs}{|D(\bu)|^3_{\eps}}DG(\bu)
			D\be
			\\
			&\leq
			\tau_s
			\int_{\Omega}
			\frac{D\bu:D\bs}{|D(\bu)|^3_{\eps}}(DG(\bu+\bh)-DG(\bu))D\be
			+
			\frac{D\bh:D\bs}{|D(\bu+\bh)|^3_{\eps}}DG(\bu+\bh)D\be
			\\
			&\leq
			\tau_s
			\eps^{-2}
			\|D\bs\|_{L^{\infty}(\Omega)} \|DG(\bu+\bh)-DG(\bu)\| \|D\be\|
			+
			\taus
			\eps^{-3}
			\|D\bh\|_{L^{\infty}(\Omega)} \|D\bs\|_{L^{\infty}(\Omega)} \|DG(\bu+\bh)\| \|D\be\|
			\\
			&\leq
			\tau_s
			\eps^{-2}
			C(1+|\ln h|)^{1/2} h^{-1} C_G  \|D\bs\| \|D\bh\| \|D\be\|
			+
			\taus
			\eps^{-3}
			C(1+|\ln h|) h^{-2} \mu^{-1} \|\blf\|_{-1}
			\|D\bh\| \|D\bs\| \|D\be\|.
		\end{align*}
		Using reverse triangle and H\"older's inequalities, noting that $\||D\bu|_{\eps}^{-1}\|_{L^{\infty}(\Omega)}\leq \eps^{-1}$, exploiting discrete Sobolev and inverse inequalities, we obtain
		\begin{align*}
			&\left|\ 
			-
			\tau_s
			\left(
			\left(\frac{1}{|D(\bu+\bh)|_{\eps}}-\frac{1}{|D\bu|_{\eps}}\right)DG'(\bu;\bs),D\be
			\right)
			\right|
			\\
			&\leq
			\tau_s
			\int_{\Omega}
			\left(\frac{|D\bh|}{|D(\bu)|_{\eps}|D(\bu+\bh)|_{\eps}}\right)DG'(\bu;\bs),D\be
			\\
			&\leq
			\tau_s
			\eps^{-2}
			\|D\bh\|_{L^{\infty}(\Omega)}\|DG'(\bu;\bs)\| \|D\be\|
			\\
			&\leq
			\tau_s
			\eps^{-2}
			C_G C(1+|\ln h|)^{1/2} h^{-1}
			\|D\bh\| \|D\bs\| \|D\be\|.
		\end{align*}
		By combining the above bounds, we get
		\begin{align*}
			\| D\be\|
			\leq
			C
			\mu^{-1}
			\tau_s
			\eps^{-2}
			(1+|\ln h|)^{1/2} h^{-1}
			\left(
			C_G 
			+
			\eps^{-1}
			(1+|\ln h|)^{1/2} h^{-1} \mu^{-1} \|\blf\|_{-1}
			\right)
			\|D\bh\| \|D\bs\| 
			=
			\hat{C}_G
			\|D\bh\| \|D\bs\| .
		\end{align*}
		In this way, $G'(\bu;\cdot)$ is Lipschitz continuous with constant $\hat{C}_G$. Since the bound holds for arbitrary $\bu$, we have that $ G $ is Lipschitz continuously differentiable on $\bV_h$ with constant $\hat{C}_G$.
	\end{proof}
	\subsection{Anderson Accelerated Picard algorithm for regularized Bingham Equations \eqref{regeqn}}
	In previous subsection, we proved that the solution operator $G$ of Picard iteration \eqref{picardbinghamV} of regularized Bingham equation satisfies 
	Assumption \ref{assume:g}. 
	To apply the one-step residual bound of \cite{PR21}, 
	 we further require satisfaction of
	Assumption \ref{assume:fg}; namely, there a constant $ \sigma >0 $ such that for any $\bu,\bs\in\bV$
	\begin{align}\label{F}
		\|F(\bu)-F(\bs)\|\geq \sigma \|\bu-\bs\|,
	\end{align}  
	where $F(\bu):=G(\bu)-\bu$.
	
	
	%
	As discussed in subsection \ref{aasection}, condition \eqref{F} can be monitored and
	enforced by a 
	safeguarding strategy for a given $\bar \sigma > 0$, although as shown in section 
	\ref{sec:numerics}, it was not necessary to do so here.
	Under these assumptions and with Lemmas \ref{LipschitzG}, \ref{lemma:wellG'}, \ref{lemma:LipschitzG'} and \ref{frechetGlemma}, Theorem \ref{thm:genm} shows the convergence of Algorithm \ref{alg:anderson} where $G$ is the solution operator of the Picard iteration for regularized Bingham equation.
	\begin{theorem}\label{thm:AAPicB}
		
		Suppose \eqref{F} holds for some $\sigma>0$, 
		and suppose the direction sines between each column $i$ of $F_j$ defined by 
		\eqref{Fk1} and the subspace spanned by the preceeding columns satisfy
		$|\sin(f_{j,i},\text{span }\{f_{j,1}, \ldots, f_{j,i-1}\})| \ge c_s >0$,
		for $j = k-m_k, \ldots, k-1$.
		Then, for any step $k>m$ 
	the following bound holds for the AA Picard residual
	\begin{align}
		\nr{w_{k+1}} & \le \nr{w_k} \Bigg(
		\theta_k (1-\beta_{k} + C_G \beta_{k})
		+ \frac{C \hat{C}_G\sqrt{1-\theta_k^2}}{2}\bigg(
		\nr{w_{k}}h(\theta_{k})
		\nonumber \\ &
		+ 2  \sum_{n = k-{m_{k}}+1}^{k-1} (k-n)\nr{w_n}h(\theta_n) 
		+ m_{k}\nr{w_{k-m_{k}}}h(\theta_{k-m_{k}})
		\bigg) \Bigg),
	\end{align}
	for residual $w_k$, where $\theta_{k}$ is the gain from the optimization problem.
\end{theorem}

\begin{remark}The direction sine condition in the hyptheses of Theorem
	\ref{thm:AAPicB} can be directly enforced by the method described in 
	\cite[Section 5.1]{PR21}.
\end{remark}

{\color{black}
\begin{remark} While $C_G$ is not proven above to be less than 1 (see Remark \ref{rem1}),
numerical tests below and in \cite{AHOV11} suggests this is typically the case, at least 
when near the solution.  So even though Theorem \ref{thm:AAPicB} does not guarantee 
global convergence, it does show that AA reduces the first order term in term in
the residual expansion in comparison to the corresponding fixed-point iteration, 
specifically $(1-\beta_{k} + C_G \beta_{k})$ to $\theta_k (1-\beta_{k} + C_G \beta_{k})$, where $\theta_k$ is the gain of the step $k$ optimization problem.  Thus, in practice, a good initial guess is expected to both keep $C_G$ small and make the higher terms negligible, which in turn makes AA improve convergence. 
\end{remark}
}

\section{Numerical Experiments}\label{sec:numerics}
This section presents the results of three numerical tests that illustrate the theory above for the Anderson Accelerated Picard iteration for regularized Bingham equations. First, we show the predicted convergence rate of the finite element discretization, and the positive impact of AA on convergence by the flow between two parallel plates, which is one of the few analytical test cases for the Bingham equations.  Then, we test Anderson Accelerated Picard iteration for regularized Bingham fluid flow on 2D and 3D driven cavity problems. Our results are in good agreement with those found in \cite{MZ01, O09}. In all numerical tests, AA provides significantly faster convergence than Picard without AA, especially with small $\eps$.  {\color{black} For all of our tests, we use $\bu_0={\bf 0}$ in the interior but also satisfying the boundary conditions of the problem.  We would expect somewhat better convergence if a better $\bu_0$ were chosen, e.g. the solution of the analogous problem with similar $\epsilon$, such as in a a continuation method.  However, with AA, initial guesses that are sufficiently bad may not perform well since the analysis suggests the higher order terms in the residual may prevent convergence.}

\subsection{Analytical test}
The flow between two parallel plates is one of the known analytical test cases for Bingham problem. In two dimensions, the analytical solutions of Stokes type Bingham equations are given by
\begin{align}\label{analiyticalsoln}
	u_1=\left\{
	\begin{array}{ll}
		\frac{1}{8} \left[(1-2\taus)^2 - (1-2\taus-2y)^2\right], & 0\leq y <\frac{1}{2}-\taus\\
		\frac{1}{8} (1-2\taus)^2 , & \frac{1}{2}-\taus\leq y \leq\frac{1}{2}+\taus\\
		\frac{1}{8} \left[(1-2\taus)^2 - (2y-2\taus-1)^2\right], & \frac{1}{2}+\taus < y \leq 1\\
	\end{array} 
	\right.,\
	u_2=0,\
	\text{and}\
	p=0.
\end{align}
The rigid (or plug) region $\{y\in\Omega\ |\frac{1}{2}-\taus\leq y \leq\frac{1}{2}+\taus\}$ is the kernel moving at constant velocity. We choose $\taus=0.3$. The discretization uses $(P_2,P_1)$ Taylor-Hood elements on a uniform triangular mesh. We take $\mu=1$ and external force $\blf=\textbf{0}$, and perform Anderson accelerated Picard iterations with depth $m=0\ \text{(no acceleration)}, 1, 2, 5$ and $10$, and will test both convergence rates for \eqref{fem} and efficiency of AA Picard solver. {\color{black} The initial guess is $\bu_0={\bf 0}$ except satisfying Dirichlet boundary conditions defined by the true analytical solution in \eqref{analiyticalsoln}.}

We display the number of iterations that reduce the relative residual of the velocity by $10^{-8}$ for varying depths $m$, mesh sizes $h$ and regularization parameters $\eps$ in Table \ref{table:iter}. When $\eps\rightarrow 0$ and/or the mesh width decreases, the required number of iterations increases, as we expect from our analysis in the previous section. Also, AA provides better convergence results as we increase the depth. This improvement is more apparent in lower values of $\eps$, which is required to obtain an accurate solution. In the case of $m=0$ (without AA), the numbers of iterations are 
substantially higher; however, with AA they decrease significantly.
The fastest convergence is obtained with depths $m=5$ and $10$, as seen in 
Table \ref{table:iter}.

\begin{table}[H]
	\centering
	\small\addtolength{\tabcolsep}{-5pt}
	\begin{tabular}{clccccc|lccccc}
		\hline
		$\downarrow$ h $\varepsilon\rightarrow$& & $10^{-1}$ &  $10^{-2}$ & $10^{-3}$  & $10^{-4}$    & $10^{-5}$  	&  & $10^{-1}$ &  $10^{-2}$ & $10^{-3}$  & $10^{-4}$    & $10^{-5}$ \\
		\hline
		&$m=0$&    &    &     &     &     &$m=1$&   &    &    &    & \\
		1/8   &     & 10 & 34 & 42  & 88  & 101 &     & 8 & 22 & 24 & 31 & 43 \\
		1/16  &     & 10 & 34 & 80  & 171 & 274 &     & 9 & 21 & 36 & 60 & 94 \\
		1/32  &     & 11 & 27 & 98  & 258 & 296 &     & 9 & 21 & 45 & 95 & 114\\
		1/64  &     & 11 & 35 & 108 & 184 & 184 &     & 9 & 21 & 46 & 57 & 49\\
		1/128 &     & 11 & 35 & 106 & 306 & 408 &     & 9 & 21 & 43 & 71 & 69  \\
		\hline
		&$m=5$&   &    &    &    &    &$m=10$&   &    &    &    & \\ 
		1/8  &     & 7 & 14 & 16 & 22 & 26 &      & 7 & 14 & 16 & 21 & 25 \\
		1/16 &     & 7 & 15 & 24 & 30 & 33 &      & 7 & 15 & 23 & 28 & 33 \\
		1/32 &     & 8 & 16 & 29 & 36 & 41 &      & 8 & 15 & 25 & 32 & 34 \\
		1/64 &     & 8 & 16 & 29 & 34 & 40 &      & 8 & 15 & 26 & 34 & 39 \\      
		1/128&     & 8 & 16 & 29 & 48 & 50 &      & 8 & 15 & 26 & 39 & 51\\       
		\hline
	\end{tabular}\caption{Number of Anderson accelerated Picard iterations required for reducing the residual by $10^{-8}$ with different depths for regularized Bingham equation when $\taus=0.3$ and  $h$ and $\eps$ is changing in analytical test case}\label{table:iter}
\end{table}
%

\subsection{2D driven cavity}
We next test Anderson accelerated Picard iteration for the regularized Bingham equations on a lid-driven cavity problem. The domain for the problem is the unit square $\Omega=(0,1)^2$ and we impose Dirichlet boundary conditions by $ \bu |_{y=1}=(1,0)^T $ and $\bu=\textbf{0}$ everywhere else. The discretization uses $(P_2,P_1)$ Taylor-Hood elements on a uniform mesh. {\color{black} Initial guess $\bu_0$ satisfies the boundary condition of the problem and $\bu_0=0$ elsewhere in the domain.}

Figure \ref{relerrplots} shows the number of iterations of the Anderson accelerated Picard iteration with varying depth $m$ and regularization parameter $\eps$, when  $h=1/64$ and yield stresses $\taus=2$ and $\taus=5$. Iterations were  run until the relative $L^2$ velocity residual fell below $10^{-8}$.  As $\eps$ becomes smaller, the required number of iterations increases. As illustrated in Figure \ref{relerrplots}, the original (unaccelerated) Picard method converges very slowly compared to the accelerated method. However, with AA, convergence is much faster.  While larger $m$ gives faster convergence, we note there is only modest gain past $m=1$ in this test.

Figure \ref{contours} shows the growth of rigid region (white) as the value of yield stress $\taus$ increases. When the rigid region enlarges, the yielded (fluid) region (shaded) remains close to the lid. These results agree well with those in \cite{MZ01,O09,BE1980}.

\begin{figure}[H]
	\centering
	\includegraphics[width = .3\textwidth, height=.25\textwidth]{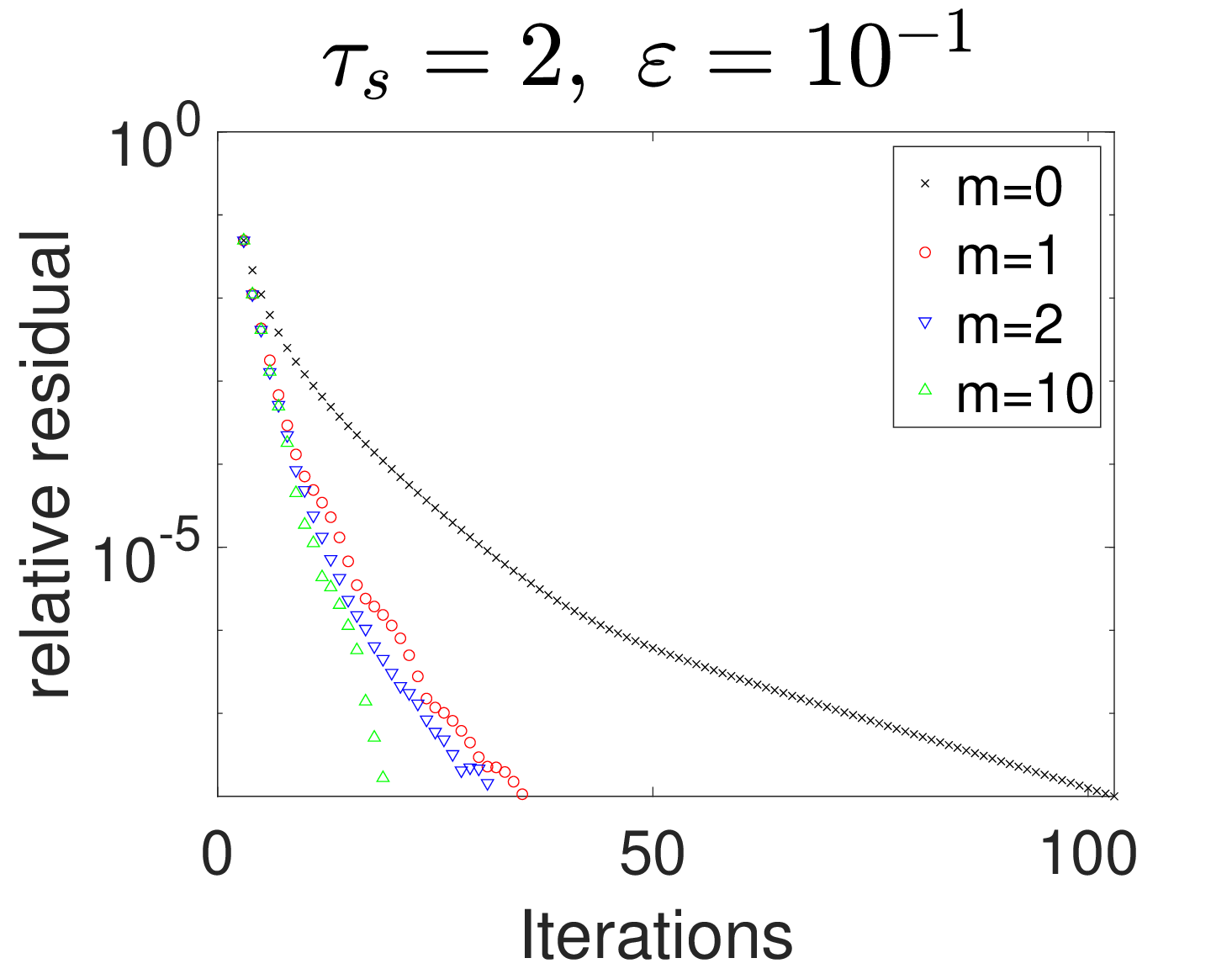}
	\includegraphics[width = .3\textwidth, height=.25\textwidth]{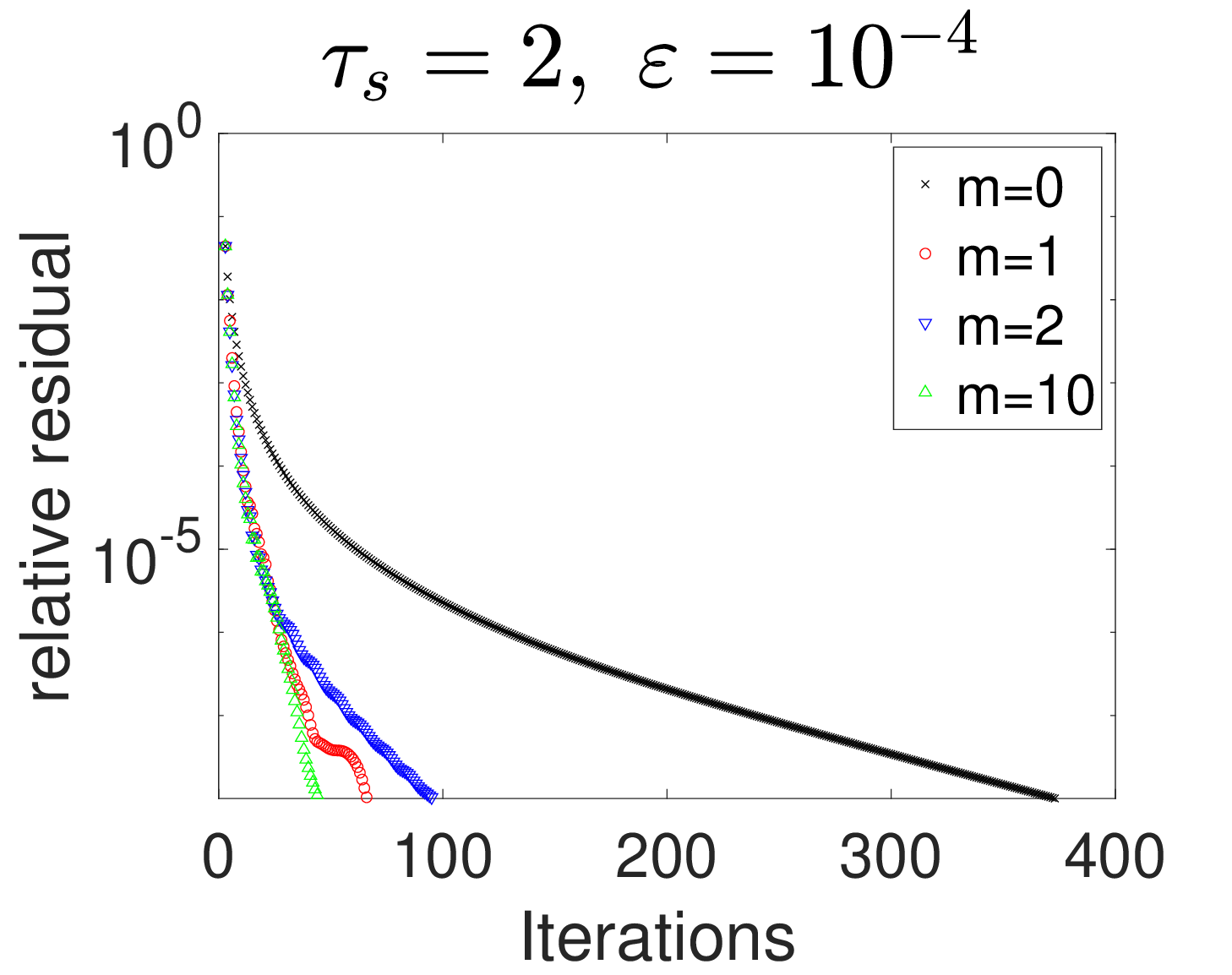}
	\includegraphics[width = .3\textwidth, height=.25\textwidth]{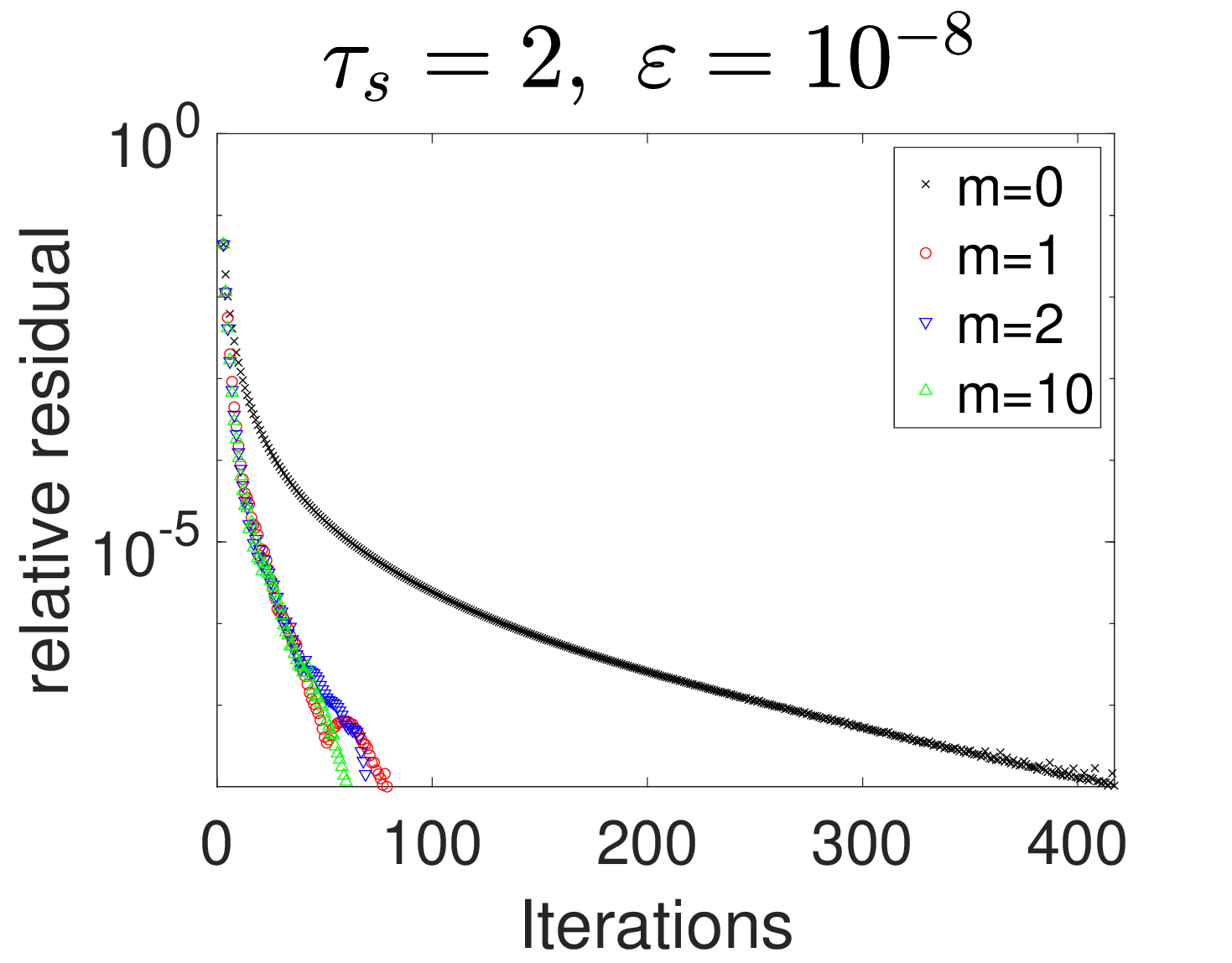}
\\
	\includegraphics[width = .3\textwidth, height=.25\textwidth]{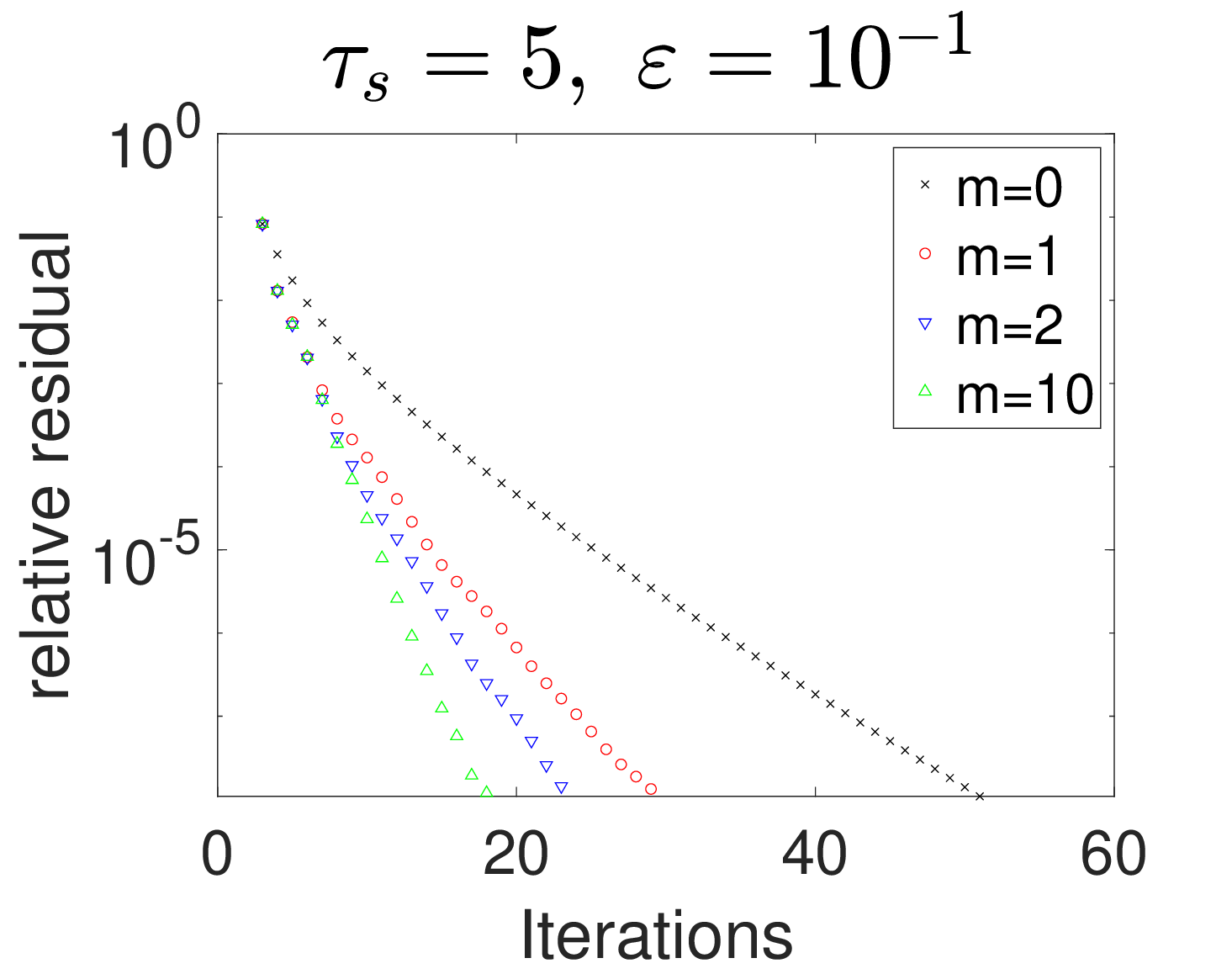}
	\includegraphics[width = .3\textwidth, height=.25\textwidth]{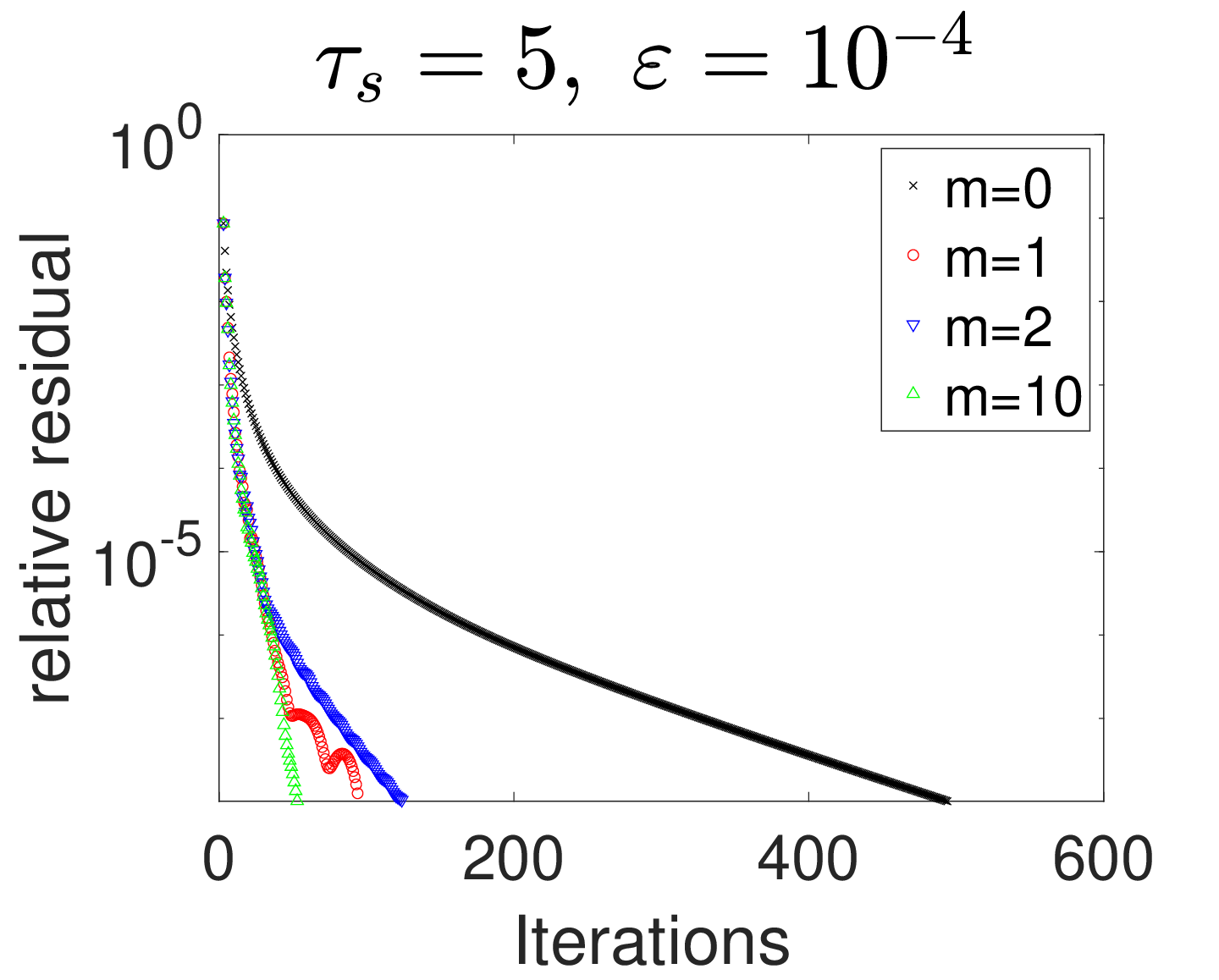}
	\includegraphics[width = .3\textwidth, height=.25\textwidth]{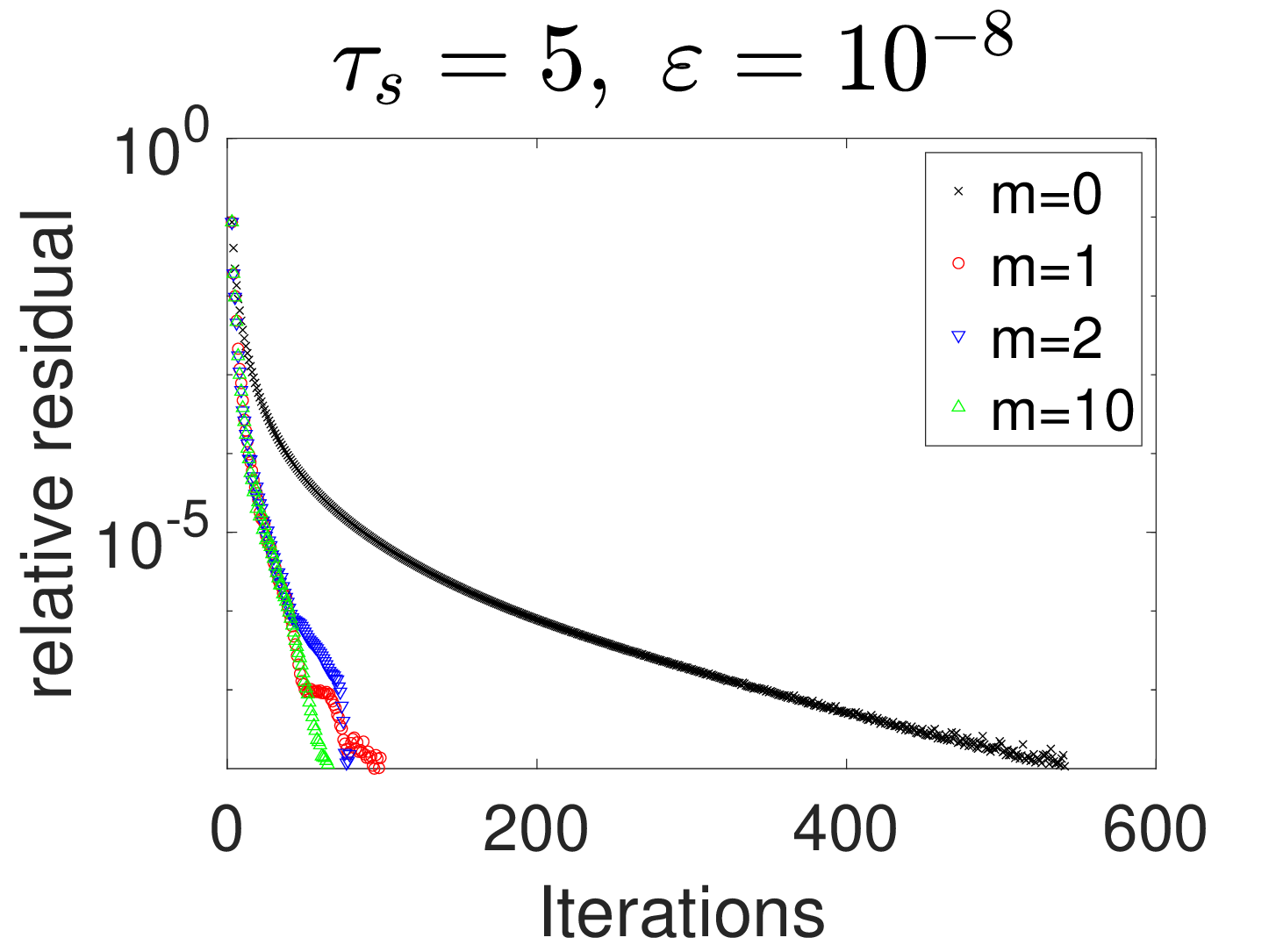}
	\caption{{\color{black}Convergence for different $\varepsilon=10^{-1}, 10^{-4},10^{-8}$(left to right), when  $\tau_s=2$ (top) and  $\tau_s=5$ (bottom)  with varying $m$.}}\label{relerrplots}
\end{figure}

\begin{figure}[H]
		\centering
	\includegraphics[width = .3\textwidth, height=.22\textwidth]{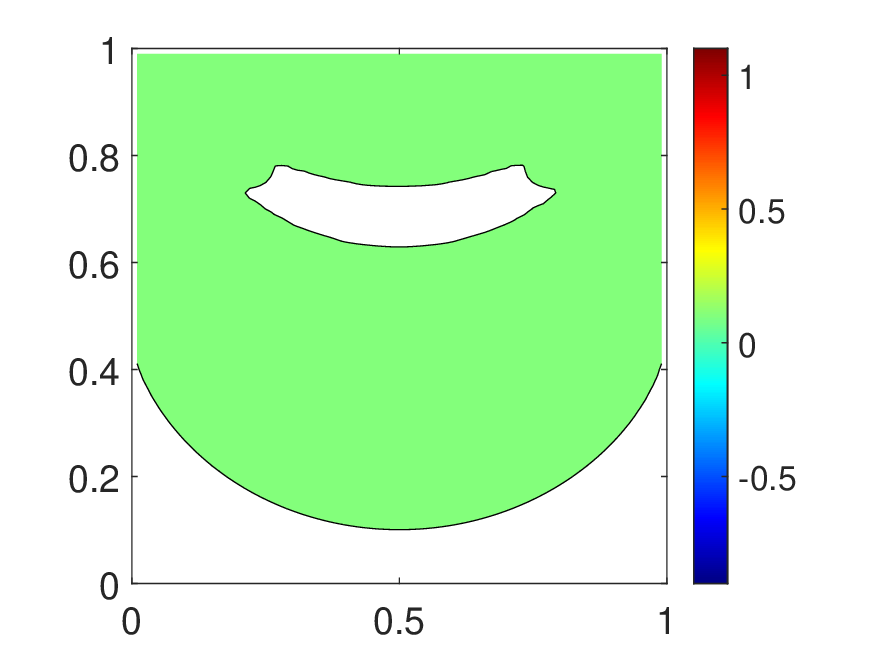}
	\includegraphics[width = .3\textwidth, height=.22\textwidth]{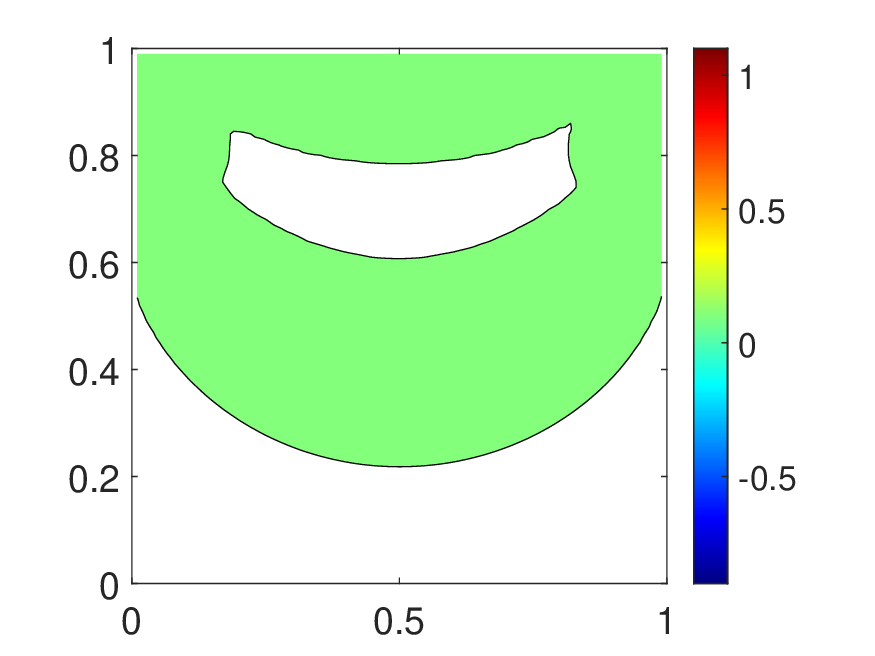}
	\includegraphics[width = .3\textwidth, height=.22\textwidth]{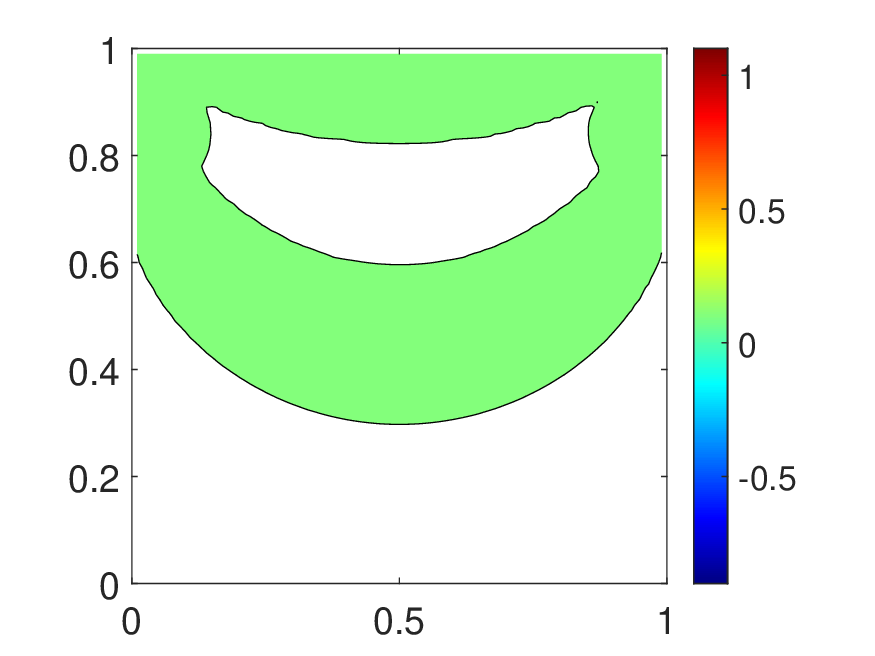}
	\caption{Growth of the rigid region (white) for lid-driven flow by $\tau_s=2,5,10$ (left to right) when $h=1/64$ and $\eps=10^{-4}$.}\label{contours}
\end{figure}

\subsection{3D driven cavity}
We now test the Anderson accelerated Picard iteration for regularized Bingham equations on the 3D lid-driven cavity. In this problem, the domain is the unit cube, there is no forcing $ (\blf = {\bf 0}) $, and homogeneous Dirichlet boundary conditions are enforced on all walls and $\bu=(1,0,0)^T$ on the moving lid. We compute with $(P_2,P_1)$ elements on Alfeld split tetrahedral meshes with 134,457 total degrees of freedom (dof) weighted towards the boundary using a Chebychev grid before tetrahedralizing. We test our scheme with varying $m$, regularization parameter $\eps$, and yield stress $\taus$. {\color{black} Initial guess $\bu_0$ satisfies the boundary condition of the problem and $\bu_0=0$ everywhere else in the domain.} Our stopping criteria is residual $\|D(\bu_k-G(\bu_k))\|\leq 10^{-5}$ or 500 iterations.


{\color{black} Figure \ref{iterationsml4} illustrates the positive impact of AA on convergence for different value of $\taus$ and $m$.}
 As $\taus$ increases, number of iterations grows since the rigid zones become larger and may completely block the flow when $\taus$ is sufficiently large. For smaller values of $\eps$, more iterations are required; however, using AA reduces the iteration counts significantly 
and enables convergence even with larger values of $\taus$.

In Figure \ref{centerlines3d}, we compare centerline x-velocities when $\eps=10^{-4}$ for varying $\taus=1, 2, 5$ and $10$ (i.e. growing rigid zones) and obtain good agreement with those found in \cite{VBL03} with P1/P1 stabilized elements and \cite{O09} with 
a finite difference method.

\begin{figure}[H]
\centering
\includegraphics[width = .5\textwidth, height=.4\textwidth]{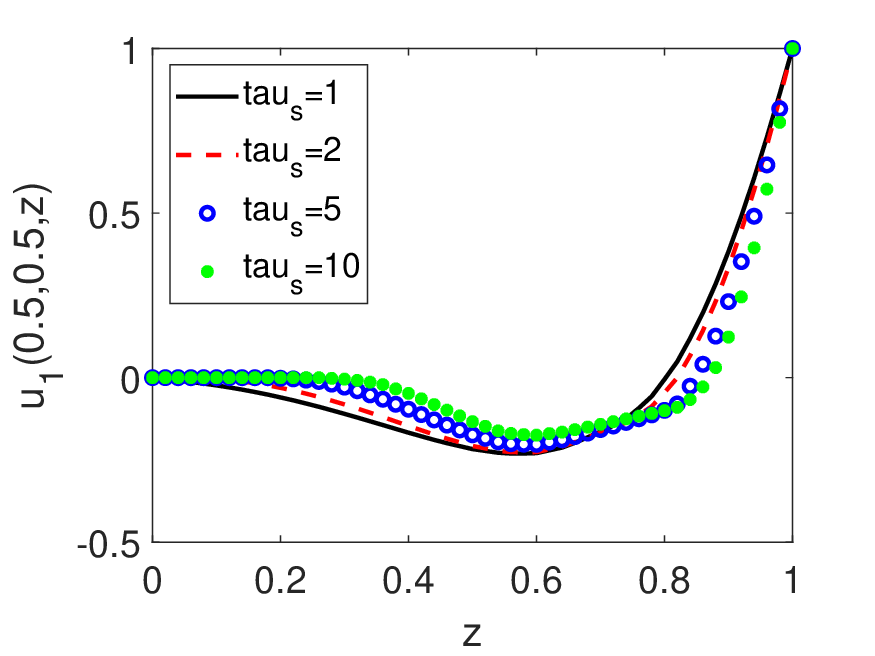}
\caption{Shown above is the centerline x-velocity plots for the 3D driven cavity simulations for $\taus=1,2,5$ and $10$, using AA Picard iteration for regularized Bingham equation with $m=10$ and $\eps=10^{-4}$, 134,457 dof.}\label{centerlines3d}
\end{figure}

\begin{figure}[H]
	\centering
	\includegraphics[width = .3\textwidth, height=.22\textwidth]{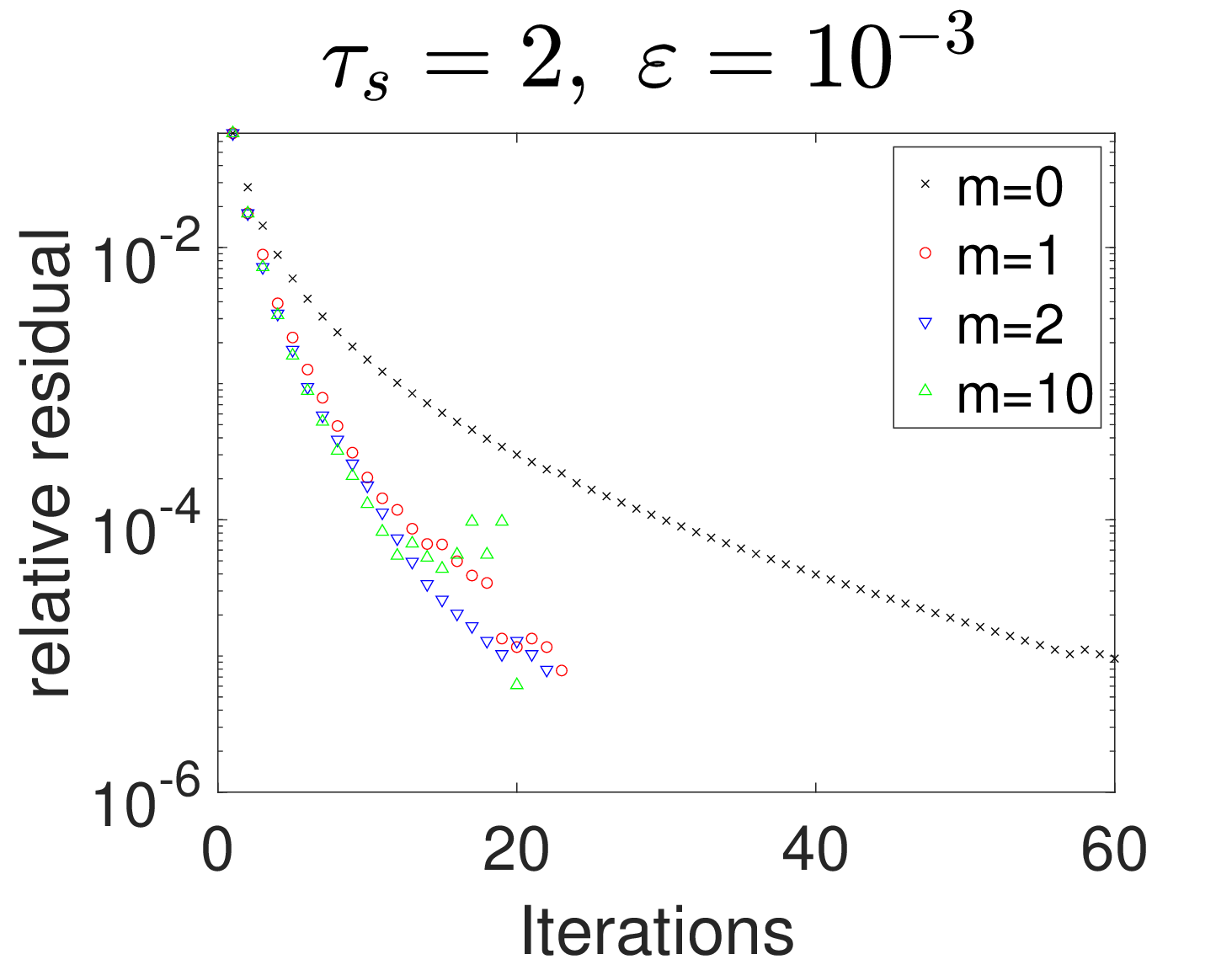}
	\includegraphics[width = .3\textwidth, height=.22\textwidth]{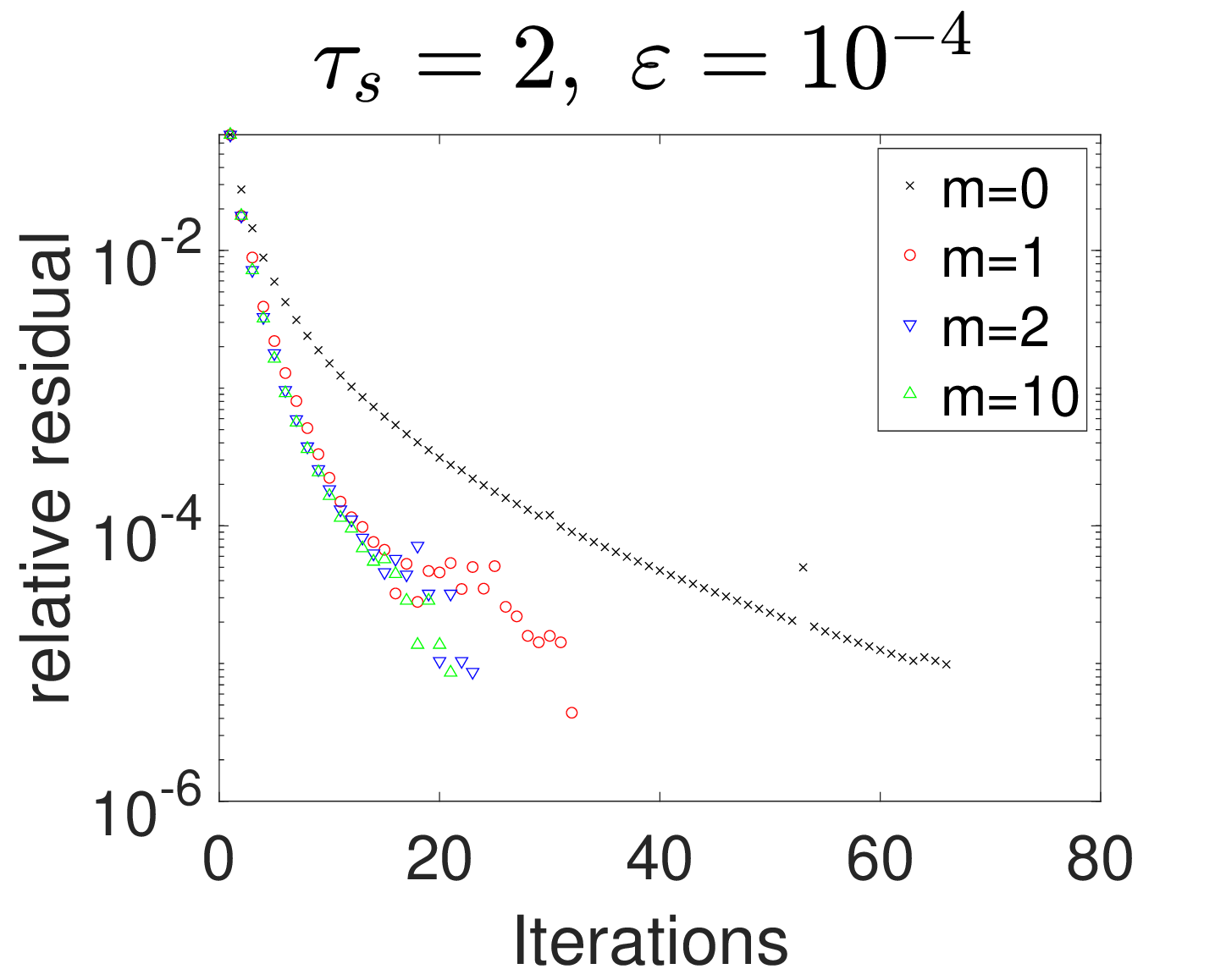}
	\includegraphics[width = .3\textwidth, height=.22\textwidth]{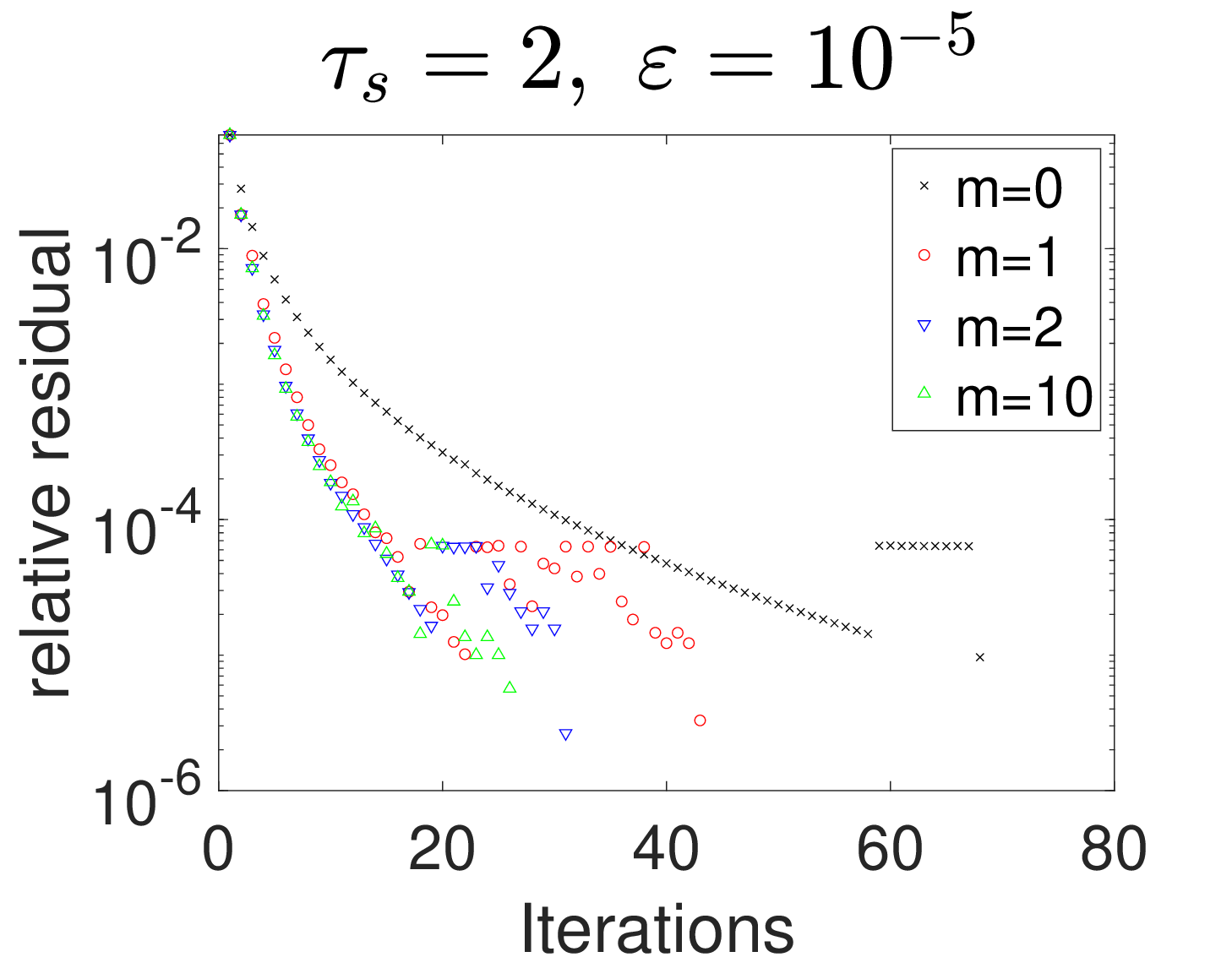}
	\\
		\includegraphics[width = .3\textwidth, height=.22\textwidth]{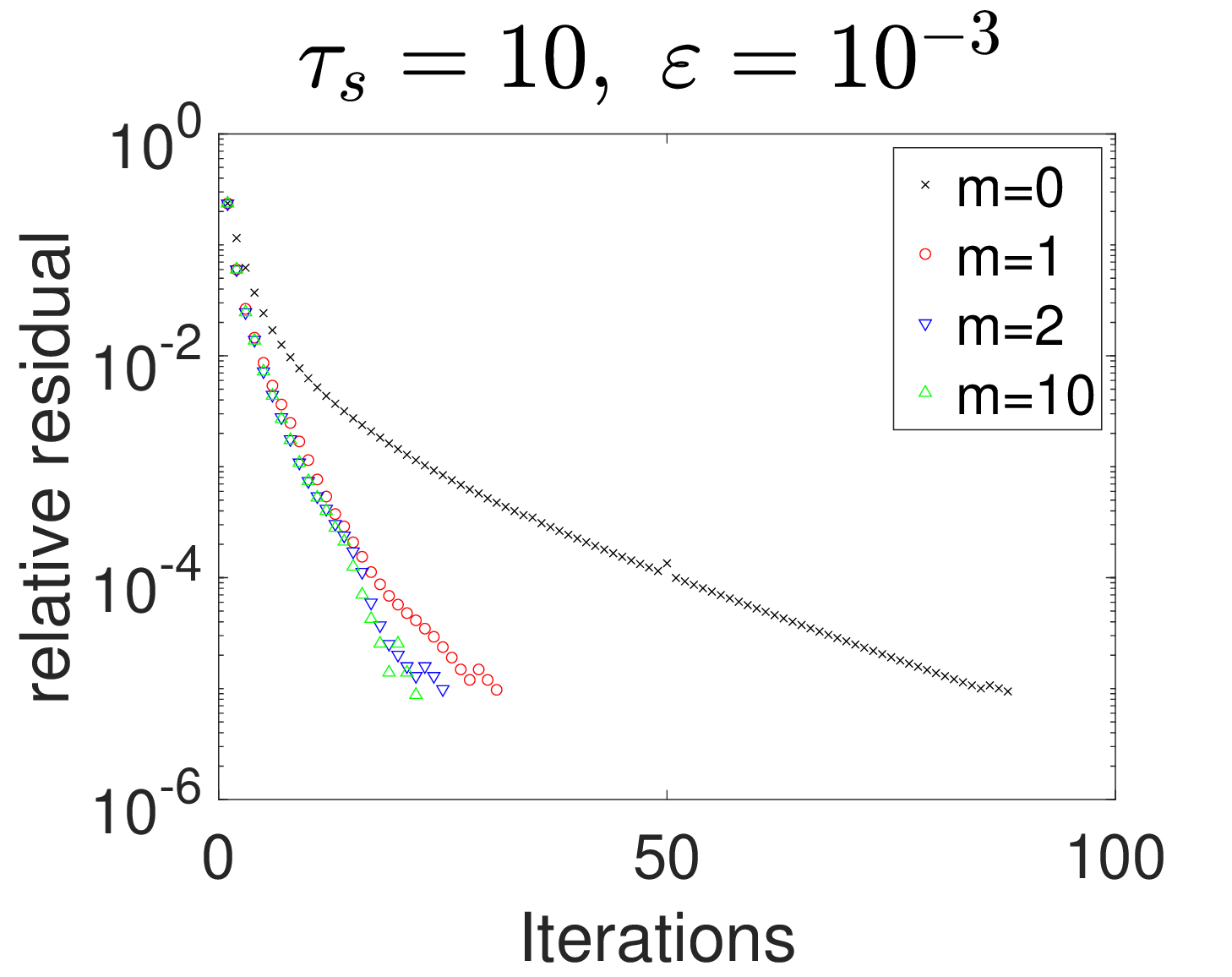}
	\includegraphics[width = .3\textwidth, height=.22\textwidth]{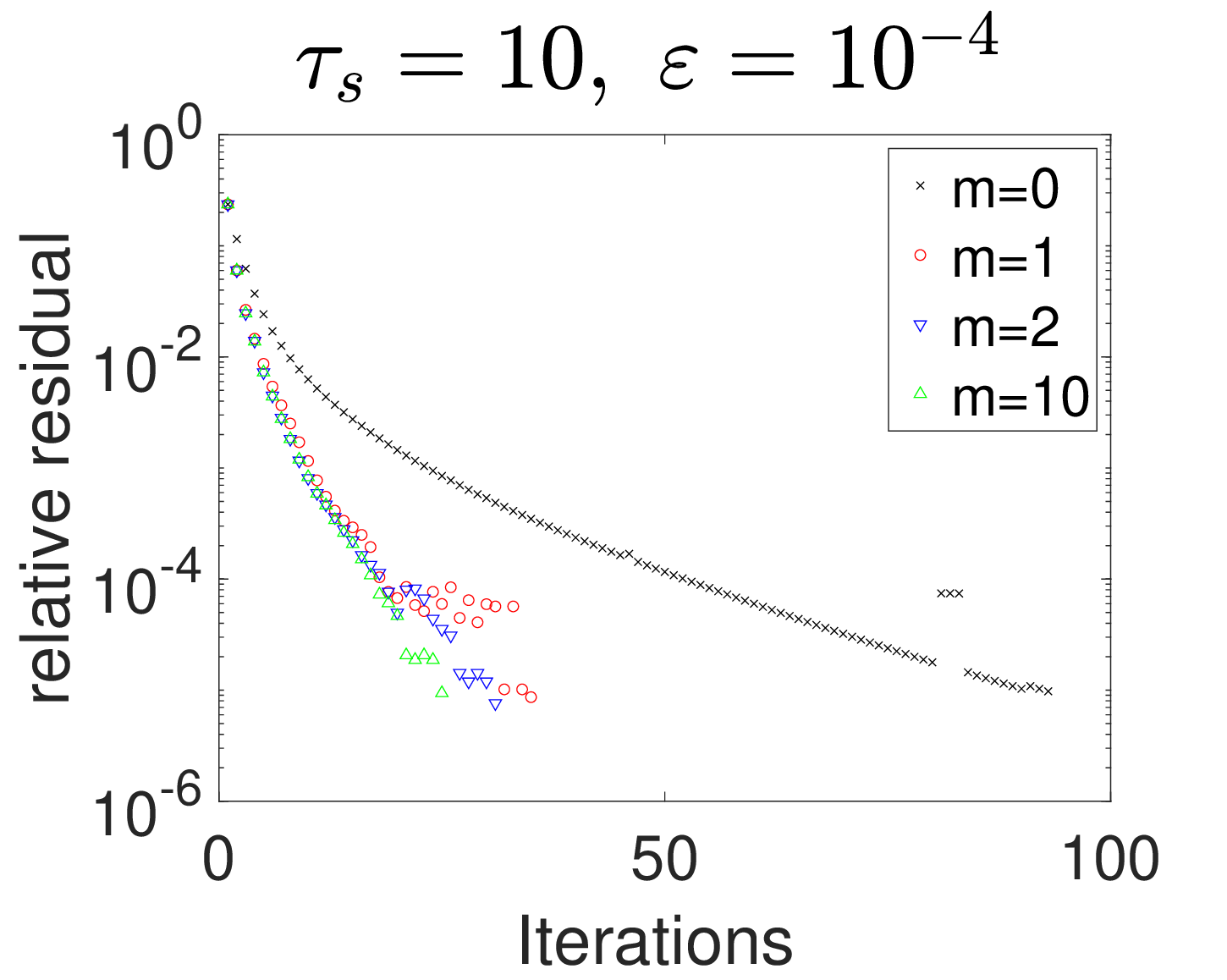}
	\includegraphics[width = .3\textwidth, height=.22\textwidth]{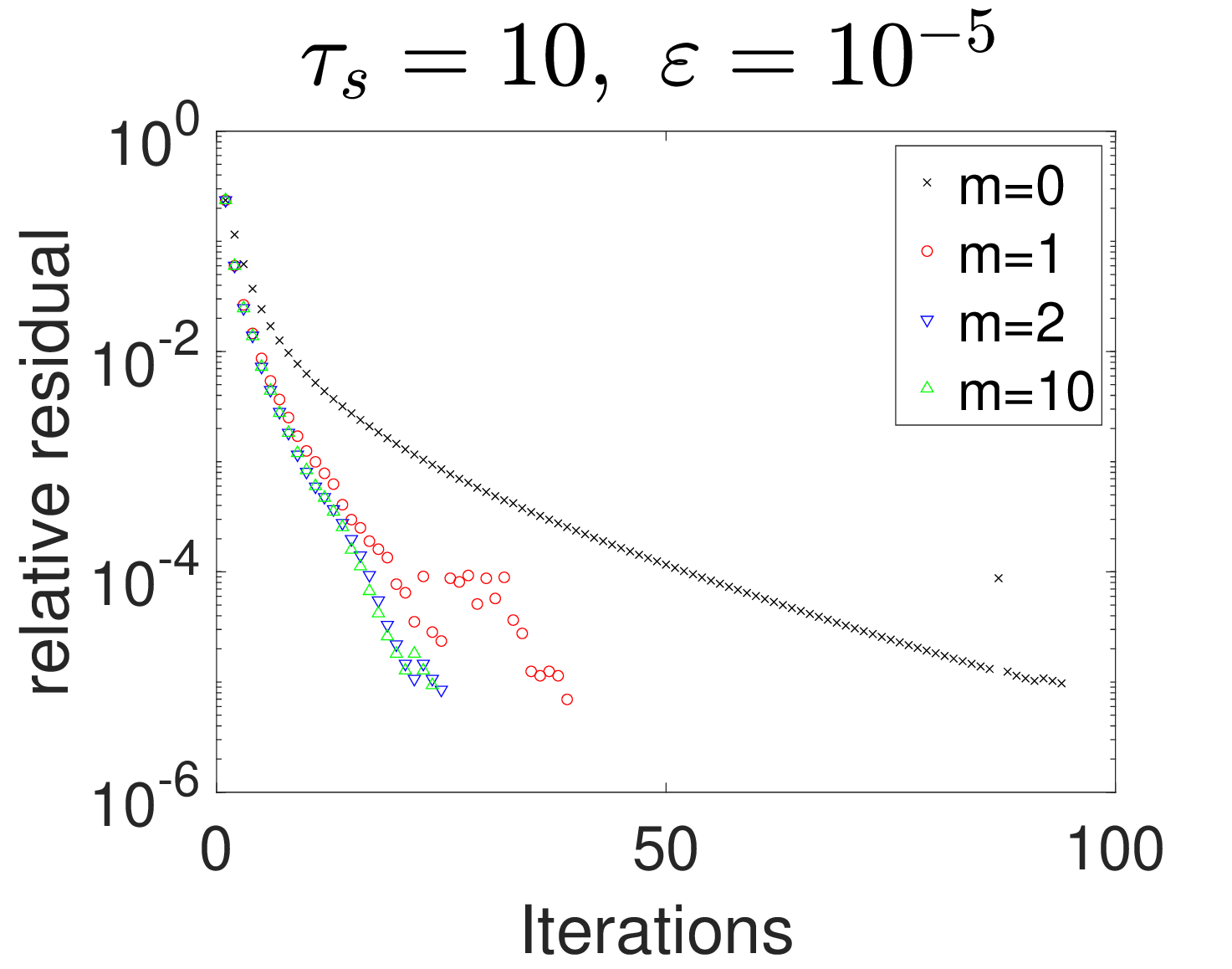}
	\caption{{\color{black}Required number of iterations, velocity residual$\leq10^{-5}$, 134,457 dof, $\tau_s=2$ (top) and $\tau_s=10$ (bottom) with varying $m$ for $\eps=10^{-3},10^{-4},10^{-5}$.}}\label{iterationsml4}
\end{figure}


%

\subsection{Numerical verification of Assumption \ref{assume:fg}}

The application of the AA theory from \cite{PR21} to the Picard iteration for the 
regularized Bingham equations relies on the satisfaction of Assumptions \ref{assume:g} 
and \ref{assume:fg}. 
We analytically verified Assumption \ref{assume:g} in section \ref{picardsection}.
Assumption \ref{assume:fg} is satisfied if the Jacobian of $g$ does not degenerate.
To demonstrate the satisfaction of this assumption, here we calculate the ratio 
$\frac{\|w_k-w_{k-1}\|}{\|u_k-u_{k-1}\|}$ for two numerical tests using constant $m$ 
to show the $\sigma_k$ (the minimum of the $k$ through $k-m_k$ ratios) is bounded well 
above 0.  Results for the analytical test problem and 2D driven cavity are shown in 
Figure \ref{sigmas}.  We observe that the ratios never get close to 0, and in 
general get larger for larger $m$.


\begin{figure}[H]
	\centering
	\includegraphics[width = .49\textwidth, height=.22\textwidth]{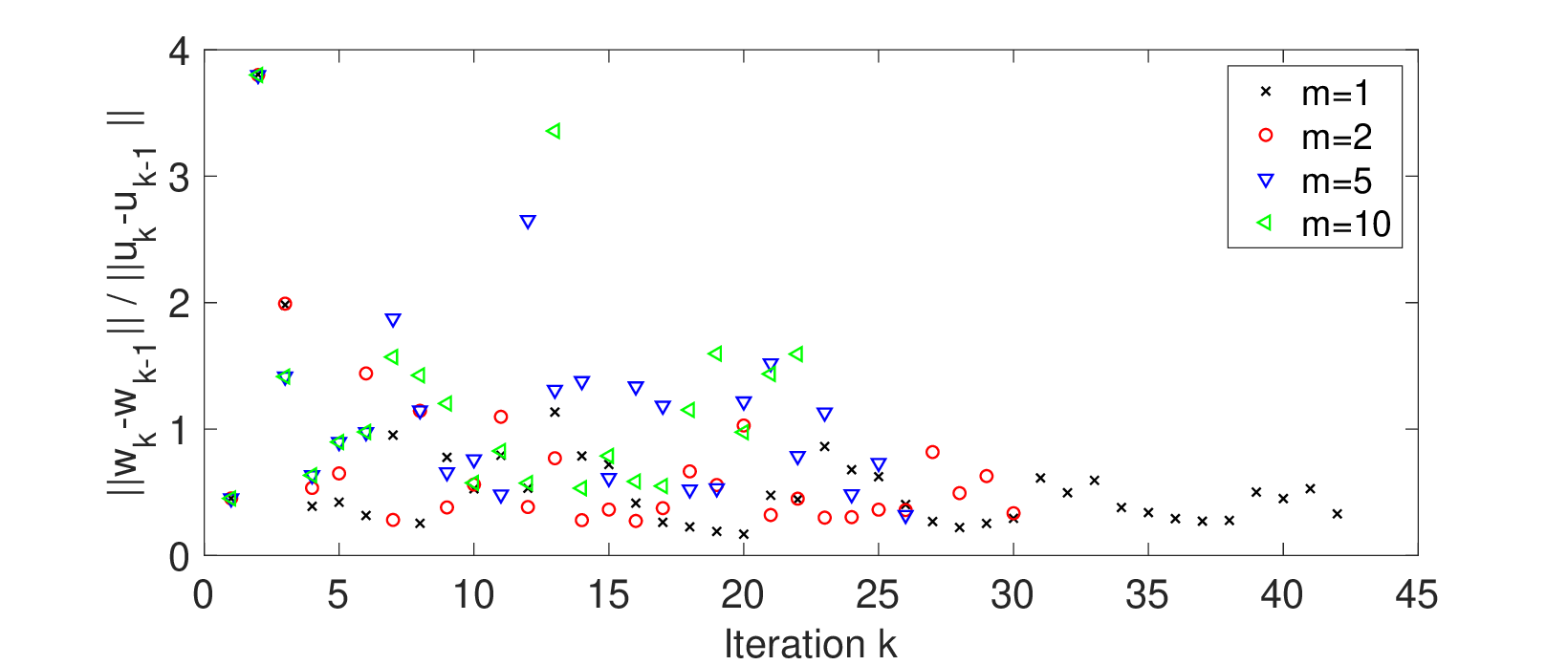}
	\includegraphics[width = .49\textwidth, height=.22\textwidth]{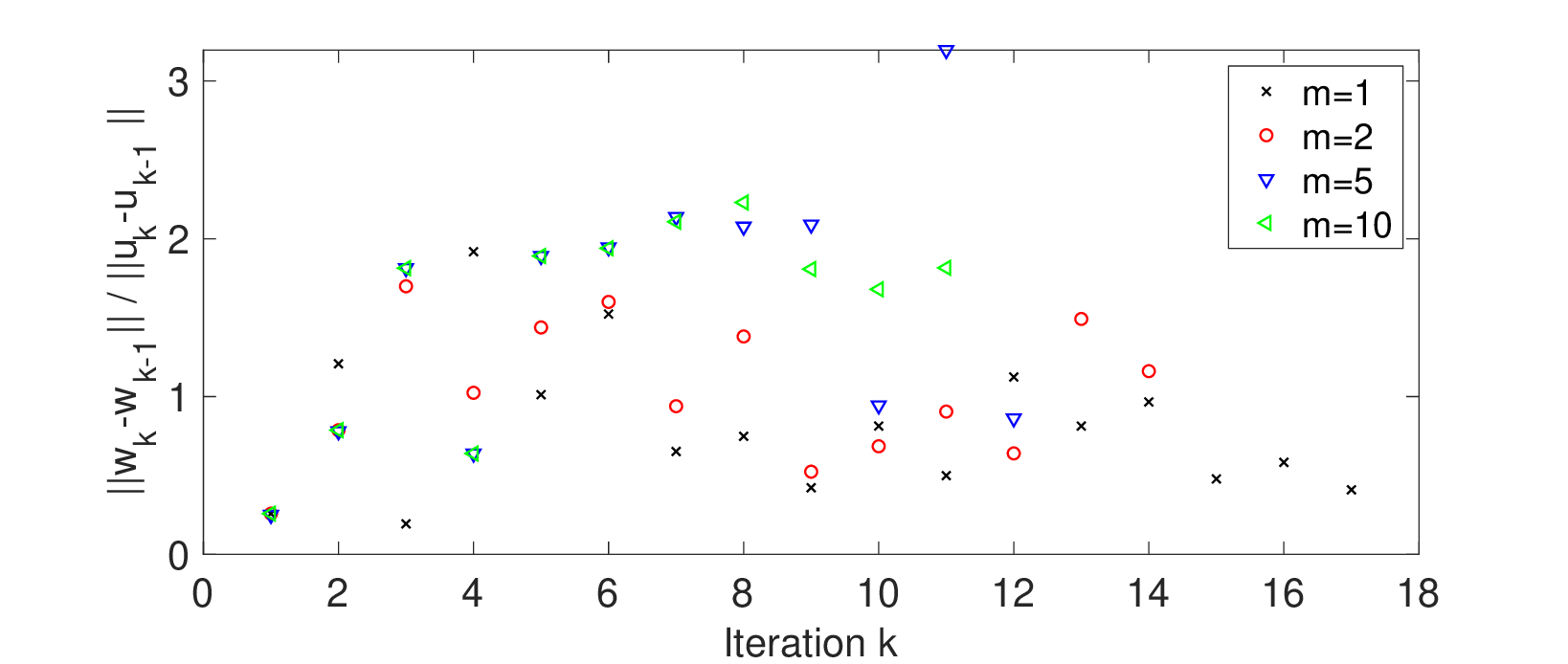}
	\caption{Shown above are ratios of $\frac{\|w_k-w_{k-1}\|}{\|u_k-u_{k-1}\|}$ for (left) the analytical test using $h=1/32$, $\tau_s=0.3$, $\varepsilon=10^{-3}$ and (right) for 2d lid driven cavity $h=1/64$, $\tau_s=2$, $\varepsilon=10^{-1}$, both with varying $m$.}\label{sigmas}
\end{figure}

\section{Convergence of the Finite Element Discretization}\label{femsection}
The convergence analysis of the numerical solutions of regularized Bingham equations by general mixed FEM does not appear well studied in the literature, and so we include here a convergence analysis for completeness.  First, we establish convergence of the velocity solution of (\ref{fem}) to the velocity of (\ref{regeqn}).

\begin{theorem}\label{convthm}
Let $(\bu,p)$ be the solution pair of the regularized Bingham problem (\ref{regeqn}).  The error in the solution $\bu_h$ to  (\ref{discdivfreeprob}) satisfies
\begin{equation}\label{firsterrbd}
	\begin{aligned}
		\|D(\bu-\bu_h)\|
		&\leq  \mu^{-1} \sqrt{d}  \inf_{q_h\in Q_h} \|p- q_h\| + 
		\left(3
		+
		3\taus^2 \eps^{-2} \mu^{-2} \right)^{1/2}\|D(\bu-I_h(\bu))\|.
	\end{aligned}
\end{equation}

Furthermore, if  $h$ is small enough so that $\| |D\bu|^{-1}_{\eps} D(I_h(\bu)-\bu_h) \| \leq1$, we also have the velocity error bound
\begin{equation}\label{seconderrbd}
	\begin{aligned}
		\|D(\bu-\bu_h)\|
		&\leq  \mu^{-1} \sqrt{d}  \inf_{q_h\in Q_h}  \|p- q_h\|
		+
		\sqrt{2} \|D(\bu-I_h(\bu))\| 
		+
		\left(2\mu^{-1}\taus \right)^{1/2} \|D(\bu-I_h(\bu))\|^{1/2}.
	\end{aligned}
\end{equation}

%
%

\end{theorem}

\begin{remark} \label{rem1}
The velocity error bound \eqref{firsterrbd} is optimal in $h$ for common choices of mixed finite elements such as Taylor-Hood and Scott-Vogelius, however it depends inversely on $\eps$ which can be small.
The bound \eqref{seconderrbd} is independent of $\eps$ but suboptimal in $h$, and it requires $h$ small enough with respect to $\eps$ so that \eqref{firsterrbd} can be invoked to produce $\| |D\bu|^{-1}_{\eps} D(I_h(\bu)-\bu_h) \| \leq1$.  A sufficient condition on $h$ to produce this bound is 
\[
\mu^{-1} \sqrt{d}  \inf_{q_h\in Q_h} \|p- q_h\| + 
\left(3
+
3\taus^2 \eps^{-2} \mu^{-2} \right)^{1/2}\|D(\bu-I_h(\bu))\| 
\le \eps.
\]
If $\eps \ll 1$, $\tau_s,\mu \sim O(1)$, $ \inf_{q_h\in Q_h} \|p- q_h\| \sim h^{s}$, and $ \|D(\bu-I_h(\bu))\| \sim h^{s}$, this reduces to $h \le O(\eps^{2/s})$. We note this is likely not a realizable condition in practice, however it is only a sufficient (and not necessary) condition that we believe pessimistic, and in our numerical tests we see no negative scaling with respect to $\eps$.
\end{remark}

\begin{cor} 
Let $(\bu,p)$ be the true solution of regularized Bingham problem satisfying $\bu\in H^3(\Omega)\bigcap\bV$ and $p\in H^2(\Omega)\bigcap Q$.  Then if $(\bX_h,Q_h)=(P_2,P_1)$ Taylor-Hood elements are used and $h$ is sufficiently small  (see Remark \ref{rem1}) , the error in velocity satisfies
\begin{equation}
	\begin{aligned}
		\|D(\bu-\bu_h)\|
		\leq 
		\min\left\{
		\mathcal{O}( h)
		,\
		\mathcal{O}( h^2\eps^{-1})
		\right\}.
	\end{aligned}
\end{equation}
If instead $\bu\in H^2(\Omega)\bigcap\bV$ and the $(P^{b}_1,P_1)$ mini element is used and $h$ is sufficiently small  (see Remark \ref{rem1}), then the bound becomes
\begin{equation}
	\begin{aligned}
		\|D(\bu-\bu_h)\|
		\leq 
		\min\left\{
		\mathcal{O}( h^{1/2})
		,\
		\mathcal{O}( h\eps^{-1})
		\right\}.
	\end{aligned}
\end{equation}

\end{cor}

\begin{proof}[Proof of Theorem \ref{convthm}]
First we will prove the bound \eqref{firsterrbd}.  The true solution $(\bu,p)$ of the regularized Bingham problem satisfies (\ref{weak}a) with $\bu\in\bV$ and $\bv=\bv_h\in \bV_h$.  Subtracting  (\ref{discdivfreeprob}) from this provides
\begin{align}\label{diffsebybilinear}
	a_{\varepsilon}(\bu,\bv_h)
	-
	a_{\varepsilon}(\bu_h,\bv_h)
	-
	b(p-q_h,\bv_h)
	=
	0,
\end{align}
which can be written as
\begin{align}\label{diffs}
	2\mu ( D\be, D\bv_h ) 
	+\tau_s
	\left(
	\frac{D\bu}{|D\bu|_{\varepsilon}} 
	- \frac{D\bu_h}{|D\bu_h|_{\varepsilon}},D\bv_h
	\right)
	-(p- q_h, \nabla\cdot\bv_h)
	=
	0,
\end{align}
where the error $\be=\bu-\bu_h$ is decomposed as $\be=\bu-I_h(\bu)+I_h(\bu)-\bu_h=\bleta+\bphi_h$ with $I_h(\bu)$ the interpolation of $\bu$ in $\bV_h$.
First, we utilize the monotonicity of $a_{\eps}$ to establish the convergence.

Let's choose $\bv_h=\bphi_h$, and add and subtract $a_{\varepsilon}(I_h(\bu),\bphi_h)$ on right hand side. Then we get
\begin{align*}
	a_{\varepsilon}(I_h(\bu),\bphi_h)
	-
	a_{\varepsilon}(\bu_h,\bphi_h)
	=
	b(p-q_h,\bphi_h)
	+
	a_{\varepsilon}(I_h(\bu),\bphi_h)
	-
	a_{\varepsilon}(\bu,\bphi_h)
	,
\end{align*}
Then, the monotonicity of $a_{\eps}$ provides
\begin{align}
	2\mu\|D\bphi_h\|^2
	\leq
	b(p-q_h,\bphi_h)
	+
	a_{\varepsilon}(I_h(\bu),\bphi_h)
	-
	a_{\varepsilon}(\bu,\bphi_h)
	. \label{CCC1}
\end{align}
Rewriting \eqref{CCC1} by expanding the $b$ and $a_{\eps}$ forms gives
\begin{align}\label{eqn:3terms}
	2\mu\|D\bphi_h\|^2
	\leq
	-2\mu (D\bleta,D\bphi_h)
	+
	\tau_s
	\left(
	\frac{D I_h(\bu)}{|D I_h(\bu)|_{\varepsilon}} - \frac{D\bu}{|D \bu|_{\varepsilon}},D\bphi_h
	\right)
	.
\end{align}
The first term of \eqref{eqn:3terms} is bounded by Cauchy-Schwarz, Korn's and Young's inequalities.
\begin{align*}
	(p- q_h, \nabla\cdot\bphi_h) 
	\leq 
	\|p- q_h\| \|\nabla\cdot\bphi_h\|
	\leq
	\sqrt{d}  \|p- q_h\| \|\nabla\bphi_h\| 
	\leq &
	C_K
	\sqrt{d}  \|p- q_h\| \|D\bphi_h\|\\
	\leq &
	\frac{3C_K^2\mu^{-1} d}{4}  \|p- q_h\|^2 +\frac{\mu}{3} \|D\bphi_h\|^2.
\end{align*}
The second term is be bounded by Cauchy-Schwarz and Young's inequalities.
\begin{align*}
	|2\mu(D\bleta,D\bphi_h)|
	\leq
	2\mu\|D\bleta\| \| D\bphi_h\|
	\leq
	3\mu\|D\bleta\|^2 + \frac{\mu}{3} \| D\bphi_h\|^2.
\end{align*}
For the last term of \eqref{eqn:3terms}, we first add and subtract $\frac{D \bu}{|DI_h(\bu)|_{\varepsilon}}$, and then apply reverse triangle inequality, H\"older's inequality ($L^{\infty}-L^2-L^2$), the upper bound
$ \| | D \cdot |_{\eps}^{-1} \|_{L^{\infty}(\Omega)} \le \eps^{-1}$ and  $||D\bu|_{\eps}^{-1} D(\bu)|\leq 1$, and Young's inequality to get	
\begin{align*}
	\tau_s
	\left(
	\frac{D I_h(\bu)}{|D I_h(\bu)|_{\varepsilon}} - \frac{D\bu}{|D \bu|_{\varepsilon}},D\bphi_h
	\right)
	&\leq
	\left| -
	\tau_s
	\left(
	\frac{D \bleta}{|D I_h(\bu)|_{\varepsilon}}
	,D\bphi_h
	\right)
	\right|
	+
	\left|
	\tau_s
	\left(
	\left(\frac{1}{|DI_h(\bu)|_{\varepsilon}}- \frac{1}{|D \bu|_{\varepsilon}}\right) D \bu,D\bphi_h
	\right)
	\right|
	\\
	&\leq
	\tau_s
	\|\ |D I_h(\bu)|^{-1}_{\varepsilon} \|_{L^{\infty}(\Omega)}
	\| D \bleta \|\
	\|D\bphi_h\|
	+
	\tau_s
	\int_{\Omega}
	\frac{|D \bleta|}{|DI_h(\bu)|_{\varepsilon}|D \bu|_{\varepsilon}}
	|D \bu| |D\bphi_h|
	\\
	&\leq
	2\tau_s
	\eps^{-1}
	\| D \bleta \|\
	\|D\bphi_h\|
	\\
	&\leq
	3\tau_s^2
	\eps^{-2}\mu^{-1}
	\| D \bleta \|^2
	+
	\frac{\mu}{3}
	\|D\bphi_h\|^2.
\end{align*}
By combining the bounds on all three terms of \eqref{eqn:3terms}, 
we obtain
\begin{align*}
	\mu\|D\bphi_h\|^2
	\leq
	\frac{3C_K^2\mu^{-1} d}{4}  \|p- q_h\|^2
	+
	\left(
	3\mu
	+
	3\taus^2 \eps^{-2} \mu^{-1}
	\right)
	\|D\bleta\|^2.
\end{align*}
Then, by taking the square root of each side and using the  triangle inequality, the bound \eqref{firsterrbd} is revealed.


Next we show the second bound, namely \eqref{seconderrbd}.  Choose $\bv_h=\bphi_h$ in (\ref{diffs}), and then add and subtract $\frac{D\bu_h}{|D\bu|_{\varepsilon}}$
from the second term to obtain 
%
\begin{align}\label{eqn:4term}
&2\mu \|D\bphi_h\|^2 
+
\tau_s\||D\bu|^{-1/2}_{\varepsilon}D \bphi_h\|^2
\nonumber \\
&=
2\mu(D\bleta,D\bphi_h)
-
\tau_s\left(\frac{D\bleta}{|D\bu|_{\varepsilon}} ,D\bphi_h\right)
-
\tau_s\left(
\left(\frac{1}{|D\bu|_{\varepsilon}} 
-
\frac{1}{|D\bu_h|_{\varepsilon}}\right)D\bu_h,D\bphi_h\right)
+(p- q_h, \nabla\cdot\bphi_h).
\end{align}
The first term on the right hand side of \eqref{eqn:4term} is bounded using Cauchy-Schwarz and Young's inequalities, by
\begin{align*}
2\mu(D\bleta,D\bphi_h)
\leq
2\mu\|D\bleta\| \| D\bphi_h\|
\leq
2\mu\|D\bleta\|^2 + \frac{\mu}{2} \| D\bphi_h\|^2.
\end{align*}
Under the assumption that $h$ is sufficiently small so that $ \| |D\bu|^{-1}_{\eps}  |D\bphi_h|\| \leq1$, the second term of 
\eqref{eqn:4term} satisfies the bound
\begin{align*}
\tau_s\left(\frac{D \bleta}{|D\bu|_{\varepsilon}} ,D\bphi_h\right)
\leq
\tau_s \| D \bleta \|  \| |D\bu|_{\varepsilon}^{-1}  D\bphi_h \|
\leq
\tau_s| \|D\bleta\|.
\end{align*}
To bound the third term of \eqref{eqn:4term}, let's first consider
\begin{align*}
\frac{1}{|D\bu|_{\varepsilon}} 
-
\frac{1}{|D\bu_h|_{\varepsilon}}
=
\frac{|D\bu_h|_{\varepsilon}-|D\bu|_{\varepsilon}}{|D\bu|_{\varepsilon}|D\bu_h|_{\varepsilon}} 
&=
\frac{|D\bu_h|^2_{\varepsilon}-|D\bu|^2_{\varepsilon}}{|D\bu|_{\varepsilon}|D\bu_h|_{\varepsilon}(|D\bu_h|_{\varepsilon}+|D\bu|_{\varepsilon})} 
\\
&
\leq
\frac{D\bu_h:D\bu_h-D\bu:D\bu}{|D\bu|_{\varepsilon}|D\bu_h|_{\varepsilon}(|D\bu_h|_{\varepsilon}+|D\bu|_{\varepsilon})} 
\\
&
=
\frac{D\be (D\bu_h + D\bu)}{|D\bu|_{\varepsilon}|D\bu_h|_{\varepsilon}(|D\bu_h|_{\varepsilon}+|D\bu|_{\varepsilon})}.
\end{align*}
So the third term of \eqref{eqn:4term} satisfies
\begin{align*}
\tau_s\left(
\left(\frac{1}{|D\bu|_{\varepsilon}} 
-
\frac{1}{|D\bu_h|_{\varepsilon}}\right)D\bu_h,D\bphi_h\right)
&\leq
\tau_s
\left(
\frac{D\bphi_h (D\bu_h + D\bu)}{|D\bu|_{\varepsilon}|D\bu_h|_{\varepsilon}(|D\bu_h|_{\varepsilon}+|D\bu|_{\varepsilon})} 
D\bu_h,D\bphi_h\right) 
\\&\,\,\,\,+
\tau_s
\left(
\frac{D\bleta (D\bu_h + D\bu)}{|D\bu|_{\varepsilon}|D\bu_h|_{\varepsilon}(|D\bu_h|_{\varepsilon}+|D\bu|_{\varepsilon})} 
D\bu_h,D\bphi_h\right) 
\\
&\leq
\tau_s
\int_{\Omega}
\frac{|D\bphi_h| |D\bu_h + D\bu|}
{|D\bu|_{\varepsilon}|D\bu_h|_{\varepsilon}(|D\bu_h|_{\varepsilon}+|D\bu|_{\varepsilon})} 
|D\bu_h| |D\bphi_h|
\\&\,\,\,\,+
\tau_s
\int_{\Omega}
\frac{|D\bleta| |D\bu_h + D\bu|} {|D\bu|_{\varepsilon}|D\bu_h|_{\varepsilon}(|D\bu_h|_{\varepsilon}+|D\bu|_{\varepsilon})} 
|D\bu_h| |D\bphi_h|
\\
&\leq
\tau_s
\int_{\Omega}
\frac{|D\bphi_h|^2 }
{|D\bu|_{\varepsilon}} 
+
\tau_s
\int_{\Omega}
\frac{|D\bleta| } {|D\bu|_{\varepsilon}} 
|D\bphi_h|
\\
&\leq
\taus \||D\bu|^{-1/2}_{\eps}D\bphi_h\|^2 +\tau_s  \|D\bleta\|,
\end{align*}
by the triangle inequality, using $ \| |D\bu|^{-1}_{\eps} D\bu \|_{L^{\infty}(\Omega)} \leq1$, and assuming 
$h$ is small enough so that $\| |D\bu|^{-1}_{\eps} D\bphi_h \| \leq1$.

The last term of \eqref{eqn:4term} can be bounded by Cauchy-Schwarz, Korn's and Young's inequalities, by
\begin{align*}
(p- q_h, \nabla\cdot\bphi_h) 
\leq 
\|p- q_h\| \|\nabla\cdot\bphi_h\|
\leq
\sqrt{d}  \|p- q_h\| \|\nabla\bphi_h\| 
\leq &
C_K
\sqrt{d}  \|p- q_h\| \|D\bphi_h\|\\
\leq &
C_K^2 \mu^{-1} d  \|p- q_h\|^2 +\frac{\mu}{2} \|D\bphi_h\|^2.
\end{align*}
Putting together the bounds of each term of \eqref{eqn:4term},
we obtain
\begin{align*}
\|D\bphi_h\|^2 
\leq
2\|D\bleta\|^2 
+
2\mu^{-1}\taus\|D\bleta\|
+
C_K^2 \mu^{-2} d  \|p- q_h\|^2. 
\end{align*}
Finally, taking square root of each side and using triangle inequality gives
\eqref{seconderrbd}.


\end{proof}
{\color{black}
\subsection{Analytical test for convergence verification}

In this subsection, we consider the same analytical test \eqref{analiyticalsoln} to study spatial convergence of numerical solution of regularized Bingham equation with the same discretization setting used in Section 4.1 (recall solver tolerance is $10^{-8}$ in the $H^1$ norm).  {\color{black}All results are obtained by AA enhanced Picard iteration with depth $m=10$, since we don't observe change in solutions from iterations with different depths.}

Table \ref{table:rateH1} shows the errors (true solution compared to numerical solution of the regularized model) and convergence rates for different choices of $h$ and $\eps$ by calculating numerical error in the $\bH^1$ norm.  We note this error includes both numerical error of the computed regularized model to the regularized model, which we prove above is at most $O(h)$ for small $\epsilon$, as well as the model consistency error, which is discussed above to be at most $O(\epsilon^{1/2})$.  Hence for larger $h$, we expect numerical error to dominate but for smaller $h$ we expect model consistency error to dominate once the numerical error is smaller than the consistency error.

The table shows that for $\epsilon=10^{-4}$, numerical error dominates until about $h=1/128$, where error no longer decreases with shrinking $h$.  The rates are very choppy for $h\ge 1/128$, but are consistent in the average with $O(h)$.  For  $\epsilon=10^{-8}$, however, numerical error appears to be dominant for all $h$'s tested (again consistent with $O(h)$ in the average), as we see error decreasing significantly from $1/128$ to $1/256$.  Hence for $h\ge 1/64$, we see only minor differences in error between the $\epsilon=10^{-4}$ and $\epsilon=10^{-8}$ solutions, but for $h\le 1/128$ we observe the $\epsilon=10^{-8}$ is better, in fact by a whole order of magnitude when $h=1/256$.

}
%

\begin{table}[H]
	\centering
	\small\addtolength{\tabcolsep}{-2.5pt}
	\begin{tabular}{c||cc||cc}
		\hline
		\multicolumn{1}{c||}{$\eps\rightarrow$}&  \multicolumn{2}{|c||}{$10^{-4}$}&    \multicolumn{2}{|c}{$10^{-8}$}\\
		\hline
		$h\downarrow$ &  $\|D(\bu-\bu_k)\|$   & Rates  &  $\|D(\bu-\bu_k)\|$ & Rates \\
		\hline
		1/4   &4.7326e-03& -      &4.7126e-03& -    \\
		1/8   &5.2061e-03&-0.14   &5.3075e-03& -0.17\\
		1/16  &8.7186e-04& 2.58   &9.3267e-04& 2.51 \\
		1/32  &6.5649e-04& 0.41   &6.6056e-04& 0.50 \\
		1/64  &1.8945e-04& 1.79   &1.4539e-04& 2.18 \\
		1/128 &1.3906e-04& 0.45   &8.7903e-05& 0.73 \\
		1/256 &1.1311e-04& 0.30   &1.2947e-05& 2.76 \\
		\hline
	\end{tabular}\caption{ $H^1$ Errors and  Convergence Rates when $\eps=10^{-4},10^{-8}$.}\label{table:rateH1}
\end{table}

\section{Conclusion}
We studied herein the acceleration of a Picard iteration to solve a finite element discretization of
the regularized Bingham equations, 
and spatial convergence of the solution to the discrete nonlinear problem.
We proved that the fixed point operator associated with the Picard iteration for the regularized Bingham equations satisfies regularity properties which allow the AA theory of \cite{PR21} to be applied, 
and thus the iteration is accelerated through the scaling of its linear convergence rate
by the gain factor of the AA optimization problem.  
We demonstrated numerically with three test problems that the Picard iteration alone may not be an effective solver for the regularized Bingham equations due to the large number of iterations required, but with AA (and in particular with $m\ge 5$), it can be an effective and efficient solver.  For the spatial convergence of the finite element discretization, we showed that optimal convergence in $h$ can be achieved but that it depends inversely on $\eps$.  We further showed that for $h$ sufficiently small with respect to $\eps$, suboptimal convergence in $h$ which is independent of $\eps$ can also be achieved.

\section{Acknowledgments}
Authors Leo Rebholz and Duygu Vargun were partially supported by National Science Foundation grant DMS 2011490. Author Sara Pollock was partially supported by National Science Foundation grant DMS 2011519.

\bibliographystyle{plain}
\bibliography{graddiv}

\end{document}